\documentclass[10pt,a4paper]{article}
\usepackage{a4wide}
\usepackage{amsfonts,amsmath,amssymb,amsbsy,amsthm}
\usepackage{latexsym}
\usepackage{comment}
\usepackage{xcolor}
\usepackage{pgf,tikz,pgfplots}
\pgfplotsset{compat=1.15}
\usepackage[
    pdftex,
    colorlinks=true,
    citecolor=blue,
    hyperindex,
    plainpages=false,
    pagebackref=true,
    bookmarksopen,
    bookmarksnumbered
]{hyperref}
\usepackage{times}
\author{NN}
\title{Reverse Alexandrov--Fenchel inequalities for zonoids}

\newcommand{\R}{\mathbb{R}}
\newcommand{\N}{\mathbb{N}}

\newcommand{\BS}{\mathbb{S}}

\newcommand{\cK}{\mathcal{K}}

\DeclareMathOperator{\lin}{lin}
\DeclareMathOperator{\spt}{supp}

\DeclareMathOperator{\MG}{G}
\DeclareMathOperator{\MA}{A}

\newtheorem{lemma}{Lemma}[section]
\newtheorem{theo}[lemma]{Theorem}

\newtheorem{conj}[lemma]{Conjecture}
\newtheorem{prop}[lemma]{Proposition}
\newtheorem{remark}[lemma]{Remark}
\newtheorem{remarks}[lemma]{Remarks}

\begin{document}

\author{ K\'aroly J. B\"or\"oczky\footnotemark[1], Daniel Hug\footnotemark[2]}
\date{}
\renewcommand{\thefootnote}{\fnsymbol{footnote}}

\footnotetext[1]{Alfr\'ed R\'enyi Institute of Mathematics, Re\'altanoda u.~13-15, H-1053 Budapest, Hungary. \newline
Central European University,
Nador utca 9, H-1051 Budapest, Hungary.
 \newline
E-mail: boroczky.karoly.j@renyi.mta.hu}

\footnotetext[2]{Karlsruhe Institute of Technology (KIT), D-76128 Karlsruhe, Germany. E-mail: daniel.hug@kit.edu}

\maketitle

\begin{abstract}
The Alexandrov--Fenchel inequality bounds from below the square of the mixed
volume $V(K_1,K_2,K_3,\ldots,K_n)$ of convex bodies $K_1,\ldots,K_n$ in $\R^n$ by the  product of the mixed volumes $V(K_1,K_1,K_3,\ldots,K_n)$ and $V(K_2,K_2,K_3,\ldots,K_n)$. As a consequence, for integers $\alpha_1,\ldots,\alpha_m\in\N$ with $\alpha_1+\cdots+\alpha_m=n$ the product   $V_n(K_1)^{\frac{\alpha_1}{n}}\cdots V_n(K_m)^{\frac{\alpha_m}{n}} $ of suitable powers of the volumes $V_n(K_i)$ of the convex bodies $K_i$, $i=1,\ldots,m$, is a lower bound for the
mixed volume $V(K_1[\alpha_1],\ldots,K_m[\alpha_m])$, where $\alpha_i$ is the multiplicity with which $K_i$ appears in the mixed volume. It has been conjectured by Ulrich Betke and Wolfgang Weil that there is a reverse inequality, that is, a sharp upper bound for the
mixed volume $V(K_1[\alpha_1],\ldots,K_m[\alpha_m])$ in terms of the product of the intrinsic volumes $V_{\alpha_i}(K_i)$, for $i=1,\ldots,m$. The case where $m=2$, $\alpha_1=1$, $\alpha_2=n-1$ has recently been settled by the present authors (2020). The case where $m=3$, $\alpha_1=\alpha_2=1$, $\alpha_3=n-2$ has been treated by Artstein-Avidan, Florentin, Ostrover (2014) under the assumption that $K_2$ is a zonoid and $K_3$ is the Euclidean unit ball. The case where $\alpha_2=\cdots=\alpha_m=1$, $K_1$ is the unit ball and $K_2,\ldots,K_m$ are zonoids has been considered by Hug, Schneider (2011).
Here we substantially generalize these previous contributions, in cases where most of the bodies are zonoids, and thus we provide further evidence supporting the conjectured reverse Alexandrov--Fenchel inequality. The equality cases in all considered inequalities are characterized. More generally, stronger stability results are established as well.


\end{abstract}

\section{Introduction}
Mixed volumes  in Euclidean space $\R^n$ are a cornerstone of the Brunn--Minkowski theory of convex bodies. They arise in a natural way as the coefficients of the polynomial expansion of the volume of a Minkowski combination of convex bodies, thus providing a far reaching generalization of the basic Steiner formula. Let $K_1,\ldots,K_n$ be (non-empty) compact convex sets (convex bodies) in $\R^n$ and write $\cK^n$ for the space of convex bodies in $\R^n$. By polarization, the mixed volume $V(K_1,\ldots,K_n)$ of $K_1,\ldots,K_n\in \cK^n$ can be defined explicitly by
$$
V(K_1,\ldots,K_n)=\frac{1}{n!}\sum_{k=1}^n(-1)^{n+k}\sum_{1\le i_1<\cdots<i_k\le n}
V_n(K_{i_1}+\cdots+K_{i_k}),
$$
where $V_n$ denotes the volume functional.
In the following, we write
$$
V(K_1[\alpha_1],\ldots,K_m[\alpha_m])=V(K_1,\ldots,K_1,\ldots,K_m,\ldots,K_m)
$$
for the mixed volume of $\alpha_i$ copies of $K_i$, $i=1,\ldots,m$, where the integers $\alpha_1,\ldots,\alpha_m\in \N_0$ satisfy $\alpha_1+\cdots+\alpha_m=n$ (if $\alpha_i=0$, then $K_i$ is omitted). We refer to \cite[Chapter 5]{Sch14} (see also \cite[Chapter 3.3]{HW2020}) for an introduction to mixed volumes, a detailed study of their properties and further references.

As functionals of convex bodies, mixed volumes provide far reaching conceptual generalizations of basic
functionals such as surface area and mean width or other intrinsic volumes. The explicit connection to the $i$th intrinsic volume $V_i$ in $\R^n$ is given by the relation
$$
V_i(K)=\frac{\binom{n}{i}}{\kappa_{n-i}}V(K [i],B^n[n-i]),\quad i\in\{0,\ldots,n\},
$$
where $K\subset\R^n$ is a convex body, $B^n$ denotes the Euclidean unit ball of $\R^n$, and $\kappa_j=\pi^{\frac{j}{2}}/\Gamma(\frac{j}{2}+1)$ is  the $j$-dimensional volume of the $j$-dimensional unit ball in $\R^j$. The intrinsic volumes can also be introduced independently of their connection to mixed volumes. For another approach, let $K$ be a convex body in $\R^n$.  We write $K|A$ for the orthogonal projection of
 $K$ onto an affine subspace $A$ of $\R^n$. If $\MG(n,i)$ is the linear Grassmannian of $i$-dimensional linear subspaces of $\R^n$ and  $\nu_{i}$ denotes the Haar probability measure on $\MG(n,i)$, then
$$
V_i(K)=\binom{n}{i}^{-1}\frac{\kappa_i\kappa_{n-i}}{\kappa_n}
\int_{\MG(n,i)}V_i(C|L)\,\nu_{i}(dL),\quad i\in\{0,\ldots,n\}.
$$
In addition, we have $V_0(K)=1$, $V_n(K)=V(K[n])=V(K,\ldots,K)=:V(K)$ is the volume (the $n$-dimensional Lebesgue measure) of $K$, and $V_i(K)$ is the $i$-dimensional Lebesgue measure $\lambda_i(K)$ of $K$ if $\dim(K)\leq i$.

 Mixed volumes  exhibit deep and surprising connections and have applications to seemingly unrelated subjects such as combinatorics, algebraic geometry or probability.
One of the fundamental results for mixed volumes is the Alexandrov--Fenchel inequality which states that
$$
V(K_1,K_2,K_3,\ldots,K_n)^2\ge V(K_1[2],K_3,\ldots,K_n)\cdot V(K_2[2],K_3,\ldots,K_n).
$$
Various approaches to this inequality have been developed (see \cite{Sch14,HW2020,CLMS,SvH,Wang} and the literature cited there). However, despite considerable effort (see Schneider \cite{Sch90a,Sch14}) and substantial recent progress (see Shenfeld, van Handel \cite{SvH,SvH2,SvH3}), the confirmation of the  conjectured classification of all equality cases by Schneider \cite{Sch85,Schneider1994a,Schneider1994b} remains a challenging open problem.
It follows from the Alexandrov--Fenchel inequality that if $K_1,\ldots,K_m$  are compact convex sets in $\R^n$  and $\alpha_1,\ldots,\alpha_m\in\N$ satisfy $\alpha_1+\cdots+\alpha_m=n$, then
$$
V(K_1[\alpha_1],\ldots,K_m[\alpha_m])\geq V_n(K_1)^{\frac{\alpha_1}n}\cdots  V_n(K_m)^{\frac{\alpha_m}n}.
$$

As a reverse counterpart of the Alexandrov--Fenchel inequality, we state the following conjecture.

\begin{conj}
\label{ReverseAlexandrovFenchel}
 If $K_1,\ldots,K_m$, $m\leq n$, are compact convex sets in $\R^n$,  and $\alpha_1,\ldots,\alpha_m\in\N$ satisfy
$\alpha_1+\cdots+\alpha_m=n$, then
$$
\binom{n }{ \alpha_1,\ldots, \alpha_m}
V(K_1[\alpha_1],\ldots,K_m[\alpha_m])\leq V_{\alpha_1}(K_1)\cdots  V_{\alpha_m}(K_m).
$$
If $\dim(K_i)\ge \alpha_i$ for $i=1,\ldots,m$, then equality holds if and only if
$\dim(K_i)=\alpha_i$ for $i=1,\ldots,m$ and the affine hulls of $K_1,\ldots,K_m$ are pairwise orthogonal.
\end{conj}

We note that $V(K_1[\alpha_1],\ldots,K_m[\alpha_m])>0$ if and only if there exist
$\alpha_i$-dimensional convex sets $C_i\subset K_i$ for $i=1,\ldots,m$ such that
$C_1+\cdots+C_m$ is $n$-dimensional.

\medskip

Conjecture \ref{ReverseAlexandrovFenchel} was suggested by  Betke and Weil \cite[Section 6]{BW} in a less explicit form and verified in the plane (but see also \cite{Artstein,Bartha} for related recent work). For $n\geq 3$, the conjecture is proved by B\"or\"oczky and Hug \cite{BH} in the particular case where $m=2$, $\alpha_1=1$ and $\alpha_2=n-1$, which provides a reverse form of the   Minkowski inequality.

\begin{theo}[Reverse Minkowski Inequality, \cite{BH}]
\label{ReverseMinkowski}
 If $K_1,K_2$ are compact convex sets in $\R^n$, then
$$
n V(K_1[1],K_2[n-1])\leq V_1(K_1)\cdot  V_{n-1}(K_2).
$$
If $\dim(K_1)\ge 1$ and $\dim(K_2)\ge n-1$, then equality holds if and only if
$\dim(K_1)=1$,  $\dim(K_2)=n-1$ and the affine hulls of $K_1,K_2$ are pairwise orthogonal.
\end{theo}

Let $\langle \cdot\,,\cdot\rangle$ denote the Euclidean scalar product.
For a convex body $K$ in $\R^n$, the support function $h(K,\cdot):\R^n\to\R$ of $K$ is defined by $h(K,u)=\max\{\langle x,u\rangle:x\in K\}$ for $u\in\R^n$.
It is a convex and (positively) homogeneous function on $\R^n$, where (positive) homogeneity means that $h(K,\lambda u)=\lambda\,h(K,u)$
for $\lambda\geq 0$ and $u\in\R^n$. It is known that any convex and homogeneous function on $\R^n$ is the support function of a uniquely determined non-empty compact convex set.
We say that a compact convex set $Z$ in $\R^n$ is a zonoid if there exists
a non-trivial finite even generating Borel measure $\varrho(Z,\cdot)$ on $\BS^{n-1}$ such that
$$
h(Z,u)=\int_{\BS^{n-1}}|\langle u,v\rangle|\, \varrho(Z,dv),\qquad u\in\R^n.
$$
In fact, we do not have to assume that zonoids are centred at the origin $o$ of $\R^n$, but could also consider arbitrary translates in the following, since all relevant functionals are translation invariant.

In this paper, we verify Conjecture~\ref{ReverseAlexandrovFenchel} if all the  compact convex sets but possibly one (say $K_1$ with multiplicity $\alpha_1$) are zonoids or if $K_1,K_2$ are arbitrary compact convex sets, $K_1$ (say) has multiplicity $\alpha_1=1$, $K_2$ has arbitrary multiplicity and the remaining bodies are zonoids. The following
partial confirmation of Conjecture~\ref{ReverseAlexandrovFenchel} admits the additional inclusion of several copies of the unit ball $B^n$. Thus we obtain generalizations of partial results available in the literature (see the subsequent comments).

\begin{theo}
\label{ReverseAlexandrovFenchelKZ}
Let integers $0\le \gamma\le\beta\le n$,  $0\leq m\leq n$ and  $\alpha_1,\ldots,\alpha_m\geq 1$ with
   $\alpha_1+\cdots+\alpha_m=n-\beta$ be given. If $Z_1,\ldots,Z_{m}$ are zonoids and $K$ is a compact convex set in $\R^n$, then
\begin{align}
&\binom{n }{\gamma, \beta-\gamma,\alpha_1,\ldots, \alpha_m}
V(K[\gamma],B^n[\beta-\gamma], Z_1[\alpha_1],\ldots,Z_{m}[\alpha_{m}])\nonumber\\
&\qquad \qquad
\leq \kappa_{\beta-\gamma}V_\gamma(K) V_{\alpha_1}(Z_1)\cdots      V_{\alpha_m}(Z_m).
\label{refeq1.3}
\end{align}
If $\dim(Z_i)\geq \alpha_i$ for $i=1,\ldots,m$ and $\dim(K)\ge {\gamma}$, then
 equality holds if and only if the affine hulls of $K,Z_1,\ldots,Z_{m}$ are pairwise orthogonal.

If $\gamma=\beta$,  $\dim(Z_i)\geq \alpha_i$ for $i=1,\ldots,m$ and $\dim(K)\ge {\gamma}$, then
 equality holds if and only if the affine hulls of $K,Z_1,\ldots,Z_{m}$ are pairwise orthogonal. Then we also have $\dim(K)=\gamma$ and $\dim(Z_i)=\alpha_i$ for $i=1,\ldots,m$.
\end{theo}

\begin{remarks}{\rm
\begin{enumerate}
\item[(a)] In the special case where $\gamma=\beta$ we thus confirm the conjecture
in the case where one of the bodies is arbitrary and  the others are zonoids.
\item[(b)] Since $B^n$ is a zonoid, Theorem \ref{ReverseAlexandrovFenchelKZ} also yields that
\begin{align*}
&\binom{n }{\gamma, \beta-\gamma,\alpha_1,\ldots, \alpha_m}
V(K[\gamma],B^n[\beta-\gamma], Z_1[\alpha_1],\ldots,Z_{m}[\alpha_{m}])\\
&\qquad \qquad
\leq V_\gamma(K)V_{\beta-\gamma}(B^n) V_{\alpha_1}(Z_1)\cdots      V_{\alpha_m}(Z_m),
\end{align*}
However, since $\kappa_{\beta-\gamma}\le V_{\beta-\gamma}(B^n)$ with strict inequality unless $\beta-\gamma=n$ (see Lemma \ref{proj} below), this is weaker than \eqref{refeq1.3}.
\item[(c)] In the special case $\gamma=1$, $\beta=n-1$ and $m=\alpha_1=1$, and writing  $Z_1=Z$, Theorem \ref{ReverseAlexandrovFenchelKZ} yields
    $$
    V(K,Z,B^n[n-2])\le \frac{\kappa_{n-2}}{n(n-1)} V_1(K)V_1(Z).
    $$
    If $K,Z$ are at least one-dimensional, then equality holds if and only if $K$ and $Z$ lie in orthogonal affine subspaces.
Since $\kappa_{n-1}V_1(K)=nV(K,B^n[n-1])$, this recovers Theorem 1.2 in \cite{Artstein}.
\item[(d)] The special case where $\gamma=0$, $\beta=n-m$ and $\alpha_1=\cdots=\alpha_{m}=1$, which involves only
zonoids with multiplicity one (except for the unit ball $B^n$), has been established in \cite[Theorem 2]{HS11}.
However, the main contribution in \cite{HS11} was the proof of the inequality
\begin{equation}\label{eqzonohs}
V(Z_1,\ldots,Z_m,B^n[n-m])\ge 2^m\kappa_{n-m}
\end{equation}
for zonoids $Z_1,\ldots,Z_m$ in $\R^n$ with isotropic generating measures (which amounts to zonoids in special positions). Here equality holds for $m\ge 2$ if and only if $Z_1=\cdots=Z_m$ is a cube of side length $2$ (for $m=1$ the inequality is an equality). From a very special case of \eqref{eqzonohs}, inequalities for intrinsic volumes which are reverse to classical ones, have been derived in \cite{HS11} (see \cite{Brazitikos,PPV} for related recent investigations). Such reverse inequalities are based on an intermediate maximization or minimization of the relevant functionals over affine equivalence classes of convex bodies, as suggested for instance already in classical work by Behrend \cite{Behrend} (see also \cite[Introduction]{HS11} and \cite{Lutwak1}).
\item[(e)] Reverse inequalities obtained by an affine optimization process as indicated in (d) or involving only affine invariant functionals  (see \cite{Lutwak1,MilmanE2020} or \cite[Chapter 10]{Sch14} and the literature cited there) had a substantial
impact on the Brunn--Minkowski theory. The major breakthrough came with Ball's reverse isoperimetric inequality \cite{Ball91}. Since then, this line of research has been explored repeatedly (see, e.g., \cite{Ball91b,Ball2003,Barthe1,Barthe2,Li,Lutwak2,Lutwak3,Lutwak4,Schmuck} and the literature cited there). It should be pointed out, however, that Conjecture \ref{ReverseAlexandrovFenchel} is not concerned with an affine invariant inequality although mixed volumes are invariant with respect to the simultaneous application of a volume preserving affine map to all convex bodies involved.
\item[(f)] Motivated by the interpretation of the classical B\'ezout inequality in algebraic
geometry in terms of  mixed volumes, Saroglou, Soprunov, Zvavitch \cite{SSZ} showed that
\begin{equation}\label{SSZinequ}
V(L_1,\ldots,L_n)V(K)\le n\, V(L_1,K[n-1])V(L_2,\ldots,L_n,K)
\end{equation}
for all convex bodies $K,L_1,\ldots,L_n\subset\R^n$, by combining an inequality due to Discant (which provides a lower bound for the relative inradius of one convex body with respect to another convex body) with the monotonicity of the mixed volumes. More general inequalities in this spirit have been
found by  Xiao \cite{Xiao}. Although (some of) these inequalities turn out to be sharp, they do not imply
the reverse Minkowski inequality by special choices of $L_i,K$. For instance, the natural choice
$K=B^n$, $L_1=:K_1$,
$L_2=\cdots=L_n=:K_2$ yields
$$
V(K_1,K_2[n-1])\le  \frac{2\kappa_{n-1}}{\kappa_n} \cdot \frac{1}{n}\, V_1(K_1)\, V_{n-1}(K_2)
$$
with
$$
\frac{2\kappa_{n-1}}{\kappa_n}\sim \sqrt{\frac{2}{\pi}}\cdot \sqrt{n}.
$$
\end{enumerate}
}
\end{remarks}

In addition to Theorem \ref{ReverseAlexandrovFenchelKZ}, which provides support to the conjectured reverse Alexandrov--Fenchel inequality, we also verify the following generalization of the reverse Minkowski inequality  from \cite{BH}, stated as  Theorem~\ref{ReverseMinkowski}.

\begin{theo}
\label{ReverseAlexandrovFenchelKKZ}
  Let $K_1,\ldots,K_m$ be compact convex sets in $\R^n$, $2\le m\le n$, and let the  integers  $\alpha_1,\ldots,\alpha_m\geq 1$ satisfy
$\alpha_1+\cdots+\alpha_m=n$. If $\alpha_1=1$ and $K_3,\ldots,K_m$ are zonoids, then
\begin{equation}\label{ineq1.5}
\binom{n }{ \alpha_1,\ldots, \alpha_m}
V(K_1[\alpha_1],\ldots,K_m[\alpha_m])\leq V_{\alpha_1}(K_1)\cdots   V_{\alpha_m}(K_m).
\end{equation}
If $\dim(K_i)\ge \alpha_i$ for $i=1,\ldots,m$, then equality holds if and only if
$\dim(K_i)=\alpha_i$ for $i=1,\ldots,m$ and the affine hulls of $K_1,\ldots,K_m$ are pairwise orthogonal.
\end{theo}

\medskip

Stability results provide strengthened versions of geometric inequalities.  The first stability forms of the Brunn--Minkowski inequality were due to Minkowski himself
(see Groemer \cite{Gro93}).
If the distance of the convex bodies is measured in terms of the so-called Hausdorff distance, then
Diskant \cite{Dis73} and Groemer \cite{Gro88} provided close to optimal stability versions
(see Groemer \cite{Gro93}). However, the natural distance is in terms of the volume of the symmetric difference, and the
essentially optimal result is due to Figalli, Maggi, Pratelli \cite{FMP09,FMP10}.
It has been slightly improved subsequently by Segal \cite{Seg12}, and even further by Kolesnikov, Milman \cite{KoM}.
The paper Eldan, Klartag \cite{ElK14} discusses ``isomorphic" stability versions of the Brunn--Minkowski inequality.

We note that the stability of the Brunn--Minkowski inequality has been considered even if the sets may not be convex.
The case $n=1$ is clarified by Frieman (see Christ \cite{Chr12}), and the case $n=2$ by van Hintum, Spink, Tiba \cite{HST2}.
If $n\geq 3$, then partial results with essentially optimal error term are due to Barchiesi, Julin \cite{BaJ17} (when one of the sets is convex) and
to Hintum, Spink, Tiba \cite{HST} (when the two sets coincide), and a stability formula (with not optimal error term) for any pair of bounded measurable sets
 is provided by
 Figalli, Jerison \cite{FiJ17}.

Schneider  \cite{Sch90b} and Martinez-Maure \cite{MaM17} provide  stability versions of the Alexandrov--Fenchel inequality if the bodies involved have $C^2_+$ boundaries. For some additional recent related stability results, see \cite{Livshyts,BH17,BFH19,BFH21}.

\medskip

In the following, we establish stability results for Theorems~\ref{ReverseAlexandrovFenchelKZ}
and \ref{ReverseAlexandrovFenchelKKZ}. As a preparation, we need the following notation.
For integers $\alpha_1,\ldots,\alpha_m\in\N$ with
$\alpha_1+\cdots+\alpha_m=n$ and for
 linear subspaces $L_1,\ldots,L_m\subset\R^n$ with $\text{dim}(L_i)=\alpha_i$,  the bracket $[L_1,\ldots,L_m]$ is defined as
$$
[L_1,\ldots,L_m]=|\det(u_1,\ldots,u_{\alpha_1},\ldots,u_{n-\alpha_m+1},\ldots,u_n)|,
$$
where $u_1,\ldots,u_{\alpha_1}$ is an orthonormal basis of $L_1$, \ldots, $u_{n-\alpha_m+1},\ldots,u_n$ is an orthonormal basis of $L_m$. We observe that $0\leq [L_1,\ldots,L_m]\leq 1$, with $[L_1,\ldots,L_m]>0$ if and only if
$L_1,\ldots,L_m$ span $\R^n$, and $[L_1,\ldots,L_m]=1$ if and only if $L_1,\ldots,L_m$ are pairwise orthogonal.

For $i=1,\ldots,n$ we denote by $r_i(K)$ the largest radius of $i$-dimensional ball contained in $K$. In particular, $2r_1(K)$ is the diameter of $K$.

We recall that  B\"or\"oczky and Hug \cite{BH} proved a stability version of the reverse Minkowski inequality (see
Theorem~\ref{ReverseMinkowski}).
We verify various stability versions of Theorems~\ref{ReverseAlexandrovFenchelKZ}
and \ref{ReverseAlexandrovFenchelKKZ} in Sections~\ref{secstability} and \ref{sec:5}. In order to avoid technical details at this point,
here we only quote the stability version of Theorem~\ref{ReverseAlexandrovFenchelKZ} in the case where all compact convex sets are zonoids.

\begin{theo}
\label{ReverseAlexandrovFenchelZstab}
Let $\varepsilon\in[0,1]$ and $2\leq m\leq n$. Let $Z_1,\ldots,Z_m$ be zonoids in $\R^n$, and let
$\alpha_1,\ldots,\alpha_m\geq 1$ be integers with $\alpha_1+\cdots+\alpha_m=n$. If
$$
0<V_{\alpha_1}(Z_1)\cdots  V_{\alpha_m}(Z_m)\le (1+\varepsilon)
\binom{n }{ \alpha_1,\ldots, \alpha_m}
V(Z_1[\alpha_1],\ldots,Z_{\alpha_m}[\alpha_m]),
$$
then there exist  $L_i\in\MG(n,\alpha_i)$ for $i=1,\ldots,m$
  such that
$$[L_1,\ldots,L_m]\geq 1-n^{10}2^{\frac{n}{2}}\sqrt{\varepsilon}
$$
and
$$
Z_i\subset Z_i|L_i+n^{\frac{9}{2}}2^{\frac{n}{2}}r_{\alpha_i}(Z_i)\sqrt{\varepsilon}\,B^n\quad\text{for } i=1,\ldots,m.
$$
\end{theo}

\section{Mixed volumes involving zonoids}

In the following, we use results from \cite[Section 5.3.3]{Sch14}.
Let $Z$ be a zonoid in $\R^n$ with centre at $o$ and (non-negative) even generating measure $\rho(Z,\cdot)$ on $\BS^{n-1}$. This means that
$$
h(Z,u)=\int_{\BS^{n-1}}|\langle u,v\rangle|\, \varrho(Z,dv),\quad u\in\R^n.
$$
The generating measure of $Z$ is uniquely determined by this relation.
It follows that
\begin{equation}
\label{volZ}
V_n(Z)=\frac{2^n}{n!}\int_{\BS^{n-1}}\cdots \int_{\BS^{n-1}}|\det(v_1,\ldots,v_n)|\, \varrho(Z,dv_1)\cdots \varrho(Z,dv_n).
\end{equation}
Given the zonoids $Z_1,\ldots,Z_n$ in $\R^n$, an application of (\ref{volZ}) to the Minkowski combination  $Z=\sum_{i=1}^n\lambda_iZ_i$, for
$\lambda_1,\ldots,\lambda_n\geq 0$, leads to
\begin{equation}
\label{volZ1Zn}
V(Z_1,\ldots,Z_n)=\frac{2^n}{n!}\int_{\BS^{n-1}}\cdots \int_{\BS^{n-1}}|\det(v_1,\ldots,v_n)|\, \varrho(Z_1,dv_1)\cdots \varrho(Z_n,dv_n).
\end{equation}
Clearly, \eqref{volZ} is recovered as a special case of \eqref{volZ1Zn}.

For $u_1,\ldots,u_k\in \BS^{n-1}$, we write $D_k(u_1,\ldots,u_k)$ to denote the $k$-volume of the parallelepiped $[o,u_1]+\cdots+[o,u_k]$, which can be computed as the absolute value of the determinant of the vectors $u_1,\ldots,u_k$ in the $k$-dimensional linear subspace
they span (or as zero, if they are linearly dependent). Alternatively, $D_k(u_1,\ldots,u_k)$ is also equal to the norm of the alternating product of the vectors $u_1,\ldots,u_k$.

More generally, let $U_1,\ldots,U_m$ be linear subspaces of $\R^n$ with $\dim(U_i)=\alpha_i\ge 1$, where  $\alpha_1+\cdots+\alpha_m=n-\beta\le n$ and
$\beta\in \{0,\ldots,n-1\}$. Then the bracket $[U_1,\ldots,U_m]_{n-\beta}$ is defined as
$$
[U_1,\ldots,U_m]_{n-\beta}=|\det(u_1,\ldots,u_{\alpha_1},\ldots,
u_{n-\beta-\alpha_m+1},\ldots,u_{n-\beta})|,
$$
where $u_1,\ldots,u_{\alpha_1}$ is an orthonormal basis of $U_1$, \ldots ,   $u_{n-\beta-\alpha_m+1},\ldots,u_{n-\beta}$ is an orthonormal basis of $U_m$, and the determinant is calculated in the subspace $U_1+\cdots +U_m$ with the induced Euclidean structure, if the linear subspaces form a direct sum of dimension $n-\beta$, and as zero otherwise. This is equivalent to the definition in \cite[Section 4.1]{SW}. Moreover, by basic (multilinear) linear algebra we have
\begin{equation}\label{eqC}
D_{n-\beta}(u_1,\ldots,u_{n-\beta})=[U_1,\ldots,U_m]_{n-\beta}D_{\alpha_1}
(u_1,\ldots,u_{\alpha_1})\cdots D_{\alpha_m}(u_{n-\beta-\alpha_m+1},\ldots,u_{n-\beta}).
\end{equation}
Next we recall from \cite[Section 5.3.3]{Sch14} a formula for the mixed volume of convex bodies and zonoids (if at least one zonoid is involved). As usual, we write $v^{(\alpha)}(\cdot,\ldots,\cdot)$ to denote the mixed volume
of compact convex sets contained in some $\alpha$-dimensional linear (or affine)  subspace. For a $k$-dimensional affine subspace $A$ of $\R^n$ with $A=L+t$ for some $t\in\R^n$ and $L\in \MG(n,k)$,  we write $A^\bot$ to denote the   linear subspace $L^\perp\in\MG(n,n-k)$, the  orthogonal complement of $L$. For $u\in \BS^{n-1}$ we set $u^\bot=({\rm lin}\,u)^\bot$, where ${\rm lin}\,u$ is the linear subspace spanned by $u$.

 Let $1\le j\le n$,
$Z_1,\ldots,Z_j\subset\R^n$ be zonoids and let $K_1,\ldots,K_{n-j}$ be general compact convex sets. Then (see \cite[Theorem 5.3.2]{Sch14})
\begin{align}\label{genmixvol}
&V(K_1,\ldots,K_{n-j},Z_1,\ldots,Z_j)\nonumber\\
&=\frac{2^j(n-j)!}{n!}\int_{\BS^{n-1}}\ldots\int_{\BS^{n-1}} D_j(u_1,\ldots,u_j)\nonumber\\
&\qquad\times v^{(n-j)}\left(K_1|\lin\{u_1,\ldots,u_j\}^\perp,\ldots,K_{n-j}|
\lin\{u_1,\ldots,u_j\}^\perp\right)\\
&\qquad\times \varrho(Z_1,du_1)\ldots \varrho(Z_j,du_j).\nonumber
\end{align}
The $j$th projection generating measure $\varrho_{(j)}(Z,\cdot)$ of a zonoid $Z$ in $\R^n$ with generating measure $\varrho(Z,\cdot)$ is a non-negative measure on the linear Grassmannian $\MG(n,j)$ of $j$-dimensional linear subspaces of $\R^n$ and defined by
\begin{equation}\label{eqA}
\varrho_{(j)}(Z,\cdot)=\frac{2^j}{j!\kappa_j}\int_{(\BS^{n-1})^j}
\mathbf{1}\{\text{lin}\{u_1,\ldots,u_j\}\in\cdot\} D_j(u_1,\ldots,u_j)\,
\rho(Z,\cdot)^{\otimes j}(d(u_1,\ldots,u_j)).
\end{equation}
Now we combine \eqref{eqC}, \eqref{genmixvol} and \eqref{eqA} to obtain a representation of mixed volumes
of general convex bodies and of zonoids counted with multiplicities. For this, let
$K_1,\ldots,K_\beta\subset\R^n$ be compact convex sets, and let $Z_1,\ldots,Z_m$ be zonoids with multiplicities $\alpha_1,\ldots,\alpha_m\ge 1$, where $\alpha_1+\cdots+\alpha_m=n-\beta$ and $\beta\in \{0,\ldots,n-1\}$.
Then we obtain
\begin{align*}
&V(K_1,\ldots,K_\beta,Z_1[\alpha_1],\ldots,Z_m[\alpha_m])\\
&=\frac{2^{n-\beta}\beta!}{n!}\int_{\BS^{n-1}}\ldots\int_{\BS^{n-1}} D_{n-\beta}(u_1,\ldots,u_{n-\beta})\nonumber\\
&\qquad\times v^{(\beta)}\left(K_1|\lin\{u_1,\ldots,u_{n-\beta}\}^\perp,\ldots,K_\beta|
\lin\{u_1,\ldots,u_{n-\beta}\}^\perp\right)\\
&\qquad\times
\varrho(Z_1,du_1)\ldots \varrho(Z_1,d u_{\alpha_1})\ldots \varrho(Z_m,du_{n-\beta-\alpha_m+1})\ldots \varrho(Z_m,d u_{n-\beta})\\
&=\frac{2^{n-\beta}\beta!}{n!}\frac{\alpha_1!\kappa_{\alpha_1}}{2^{\alpha_1}}\cdots
\frac{\alpha_m!\kappa_{\alpha_m}}{2^{\alpha_m}}
\int_{\MG(n,\alpha_1)}\ldots\int_{\MG(n,\alpha_m)} [U_1,\ldots,U_m]_{n-\beta}\nonumber\\
&\qquad\times v^{(\beta)}\left(K_1|(U_1+\cdots+U_m)^\perp,\ldots,K_\beta|
(U_1+\cdots+U_m)^\perp\right)\\
&\qquad\times
\varrho_{(\alpha_1)}(Z_1,dU_1)\ldots  \varrho_{(\alpha_m)}(Z_m,dU_m),
\end{align*}
hence
\begin{align}\label{metaref1}
&V(K_1,\ldots,K_\beta,Z_1[\alpha_1],\ldots,Z_m[\alpha_m])\nonumber\\
&=\binom{n}{\beta,\alpha_1,\ldots,\alpha_m}^{-1}\kappa_{\alpha_1}\cdots\kappa_{\alpha_m}
\int_{\MG(n,\alpha_1)}\ldots\int_{\MG(n,\alpha_m)} [U_1,\ldots,U_m]_{n-\beta}\nonumber\\
&\qquad\times v^{(\beta)}\left(K_1|(U_1+\cdots+U_m)^\perp,\ldots,K_\beta|
(U_1+\cdots+U_m)^\perp\right)\\
&\qquad\times
\varrho_{(\alpha_1)}(Z_1,dU_1)\ldots  \varrho_{(\alpha_m)}(Z_m,dU_m).\nonumber
\end{align}
In the special case where $K_1=\cdots=K_\beta=B^n$, we obtain
\begin{align}
&V(Z_1[\alpha_1],\ldots,Z_m[\alpha_m],B^n[\beta])\nonumber\\
&=\binom{n}{\beta,\alpha_1,\ldots,\alpha_m}^{-1}
\kappa_\beta\kappa_{\alpha_1}\cdots\kappa_{\alpha_m}
\int_{\MG(n,\alpha_1)}\ldots\int_{\MG(n,\alpha_m)} [U_1,\ldots,U_m]_{n-\beta}\nonumber\\
&\qquad\times
\varrho_{(\alpha_1)}(Z_1,dU_1)\ldots  \varrho_{(\alpha_m)}(Z_m,dU_m).\label{veryspecial}
\end{align}
Since $\kappa_{n-i}V_i(K)=\binom{n}{i}V(K[i],B^n[n-i])$, the very special case $m=1$ of
\eqref{veryspecial} shows that
\begin{equation}\label{eq2}
V_{\alpha_j}(Z_j)=\kappa_{\alpha_j}\rho_{\alpha_j}(Z_j,\MG(n,\alpha_j)).
\end{equation}
Using \eqref{veryspecial},
\begin{equation}\label{Number1}
[U_1,\ldots,U_m]_{n-\beta}\le 1
\end{equation}
and \eqref{eq2}, we get
\begin{equation}
\label{zonolate}
V(Z_1[\alpha_1],\ldots,Z_m[\alpha_m],B^n[\beta])
 \le \binom{n}{\beta,\alpha_1,\ldots,\alpha_m}^{-1}
\kappa_\beta V_{\alpha_1}(Z_1)\cdots V_{\alpha_m}(Z_m).
\end{equation}
The inequality \eqref{zonolate} is sharp. Suppose that $\text{dim}(Z_j)\ge \alpha_j$ for $j=1,\ldots,k$. Then equality holds if and only if the zonoids $Z_1,\ldots,Z_k$ lie in orthogonal linear subspaces of dimensions $\alpha_1,\ldots,\alpha_k$. In fact, this inequality and the description of the equality case is a special case of  Theorem~\ref{ReverseAlexandrovFenchelKZ}.

\section{Proofs of Theorem~\ref{ReverseAlexandrovFenchelKZ} and Theorem~\ref{ReverseAlexandrovFenchelKKZ}}
\label{secKZKKZ}

We start with an important observation which allows us to compare the intrinsic volume $V_\alpha(K)$ of a convex body $K$ to the intrinsic volume $V_\alpha(K|A)$ of a projection of $K$ to a subspace $A$ of dimension $\beta\ge \alpha$ (see Lemma~\ref{proj}) and to characterize the equality case.
In the next section, when the stability
of Theorem~\ref{ReverseAlexandrovFenchelKZ} and Theorem~\ref{ReverseAlexandrovFenchelKKZ} is discussed, we establish a stronger version in the special case where $\alpha=\beta$.

For a compact convex set $K$, $u\in \BS^{n-1}$ and $H=u^\bot\in\MG(n,n-1)$, the Steiner symmetral $S_HK$ of $K$ is defined as
$$
S_HK=\left\{x+\mbox{$\frac{t-s}2$}\,u:\,x\in K|H\mbox{ and } x+tu,\;x+su\in K\right\}.
$$
It is known (see Schneider \cite{Sch14}) that $S_HK$ is a compact convex set symmetric through $H$ and $S_HK|H=K|H$. Moreover, it is known that Steiner symmetrization
does not increase the intrinsic volumes, that is,
\begin{equation}
\label{steiner}
V_\alpha(S_HK)\leq V_\alpha(K)\quad \text{for }\alpha=1,\ldots,n.
\end{equation}
This can be found in  \cite[Satz 18.5]{Leicht} or in Hadwiger's monograph \cite{Hadwiger}.
In addition,
\begin{equation}
\label{steinerequa}
V_n(S_HK)=V_n(K).
\end{equation}
For a subset $X\subset \R^n$, we write ${\rm lin}\,X$ and ${\rm conv}\,X$ to denote the linear hull and the convex hull of $X$, respectively,
and
for a compact convex set $K$, we write  ${\rm aff}\,K$ and ${\rm relint}\,K$ to denote the affine hull and the relative interior with respect to the affine full of $K$, respectively.

Let $\MA(n,\beta)$ denote the affine Grassmannian of $\beta$-dimensional affine subspaces ($\beta$-flats, for short)
of $\R^n$.

\begin{lemma}
\label{proj}
Let $1\leq \alpha\leq\beta\leq n-1$. Let $K\in\cK^n$ and $A\in \MA(n,\beta)$.  Then the following is true.
\begin{enumerate}
\item[{\rm (a)}] $V_\alpha(K|A)\le V_\alpha(K)$.
\item[{\rm (b)}] If $\dim(K)\ge \alpha$, then $V_\alpha(K|A)=V_\alpha(K)$ if and only if $K\subset A+z$ for some $z\in \R^n$ (even $K=K|A$ provided $A\cap K\neq \emptyset$).
\end{enumerate}
\end{lemma}

\begin{proof}
For the proof we can assume that $K$ has its centroid at the origin $o$ and $A\in \MG(n,\beta)$.

In a first step, we consider the case where $\beta=n-1$.

(a) Recall that $V_\alpha(S_AK)\le V_\alpha(K)$ (see \cite[Satz 18.5]{Leicht}). Since $K|A\subset S_A K$, we get
\begin{equation}\label{refineq1}
V_\alpha(K|A)\le V_\alpha(S_AK)\le V_\alpha(K).
\end{equation}

(b) For the non-trivial direction, we assume that  $V_\alpha(K|A)=V_\alpha(K)$. Then \eqref{refineq1} implies that
\begin{equation}\label{refeq2}
V_\alpha(K|A)= V_\alpha(S_AK)= V_\alpha(K).
\end{equation}
Aiming at a contradiction, we assume that $K\not\subset A$. Suppose that $\dim(K)=\alpha$. Then $V_\alpha(K|A)= V_\alpha(K)>0$ by \eqref{refeq2}, hence
$\dim(K|A)=\alpha$. Since $K\not\subset A$, it follows that $V_\alpha(K)=\lambda_\alpha(K)>\lambda_\alpha(K|A)=V_\alpha(K|A)$, which contradicts \eqref{refeq2}.
This shows that $\dim(K)\ge\alpha+1$.
   Since $0\le n-\dim(K+A^\perp)<n-\alpha<n$ and $V_\alpha(S_AK)= V_\alpha(K)$, \cite[Satz 18.5, (444b)]{Leicht} now yields that $K$ is symmetric with respect to $A$ (here we use that $K$ has its centroid at the origin), in particular we have $K|A=K\cap A\subsetneqq K$ since $K\not\subset A$. Since $\dim(K)>\alpha$, there is an open set of affine flats $E\in\MA(n,n-\alpha)$ for which $K\cap E\neq\emptyset$ and $(K|A)\cap E=K\cap A\cap E=\emptyset$. Let $\mu_{n-\alpha}$ denote the suitably normalized
motion invariant Haar measure on $\MA(n,n-\alpha)$ and let $\chi$ denote the Euler characteristic. Then the Crofton formula (see \cite[Chapter 5]{HW2020} , \cite[Section 4.4]{Sch14}) implies that
\begin{align*}
V_\alpha(K)&=c_{n,\alpha}\int_{\MA(n,n-\alpha)}\chi(K\cap E)\, \mu_{n-\alpha}(dE)\\
&>c_{n,\alpha}\int_{\MA(n,n-\alpha)}\chi((K|A)\cap E)\, \mu_{n-\alpha}(dE)\\
&=V_\alpha(K|A),
\end{align*}
where $c_{n,\alpha}$ is a constant depending only on $n,\alpha$. This is the required contradiction.

Now we turn to the general case.

(a) We choose $A_i\in\MG(n,i)$ for $i=\beta,\ldots,n-1$ such that $A=A_\beta\subsetneqq A_{\beta+1}\subsetneqq\cdots\subsetneqq A_{n-1}$. By the first step, we have $V_\alpha(K|A_{n-1})\le V_\alpha(K)$.
Since $K|A_{n-2}=(K| A_{n-1})|A_{n-2}$, we can apply the result from the first step in $A_{n-1}$ to $K|A_{n-1}$ with respect to $A_{n-2}$ to get
$$
V_\alpha(K|A_{n-2})=V_\alpha((K|A_{n-1})|A_{n-2})\le V_\alpha(K|A_{n-1})\le V_\alpha(K).
$$
Observe that by the intrinsic nature of $V_\alpha$ the respective ambient space need not be indicated. Using
$K|A_{n-3}=(K| A_{n-2})|A_{n-3}$, we obtain in the same way that
$$
V_\alpha(K|A_{n-3})=V_\alpha((K|A_{n-2})|A_{n-3})\le V_\alpha(K|A_{n-2})\le  V_\alpha(K|A_{n-1})\le V_\alpha(K).
$$
By induction, we thus get
\begin{equation}\label{refineq2}
V_\alpha(K|A)=V_\alpha(K|A_\beta)\le V_\alpha(K|A_{\beta+1})\le \ldots \le V_\alpha(K|A_{n-1})\le V_\alpha(K).
\end{equation}

(b) Now we assume that $V_\alpha(K|A)=V_\alpha(K)$. From \eqref{refineq2}, we deduce that
$$
V_\alpha(K|A)=V_\alpha(K|A_\beta)= V_\alpha(K|A_{\beta+1})= \ldots = V_\alpha(K|A_{n-1})= V_\alpha(K).
$$
By the result of the first step and $V_\alpha(K|A_{n-1})\le V_\alpha(K)$ we conclude that $K\subset A_{n-1}$, hence $K|A_{n-1}=K$. Since $K|A_{n-2}=(K|A_{n-1})|A_{n-2}$ and $ V_\alpha(K|A_{n-2})=V_\alpha(K|A_{n-1})$, it follows from the first step, applied in $A_{n-1}$, that $K=K|A_{n-1}\subset A_{n-2}$. Continuing inductively in this way it follows that $K\subset A_\beta=A$, as asserted.
\end{proof}

For $u_1,\ldots,u_k\in\R^n$, we write ${\rm lin}\{u_1,\ldots,u_k\}^\bot$ to denote the orthogonal complement of
the linear hull of $u_1,\ldots,u_k$.

\medskip

For the proof of Theorems \ref{ReverseAlexandrovFenchelKZ} and \ref{ReverseAlexandrovFenchelKKZ} we start with some preparation.

\medskip

Note that Theorem \ref{ReverseAlexandrovFenchelKZ} holds trivially with equality if $\beta=n$. Since Theorem \ref{ReverseAlexandrovFenchelKKZ} boils down to Theorem
\ref{ReverseMinkowski} if no zonoids are involved, it is sufficient in the following to consider the following range for the involved parameters.

Let $0\le \gamma_1,\gamma_2$, $\gamma:=\gamma_1+\gamma_2$, $\alpha_1,\ldots,\alpha_m\ge 1$, $\gamma\le \beta\le n-1$ such that
$\alpha_1+\cdots+\alpha_m=n-\beta$, $m\ge 1$. Let $K_1,K_2\subset\R^n$ be arbitrary compact convex sets and let $Z_1,\ldots,Z_m\subset\R^n$ be zonoids.

A special case of \eqref{metaref1} gives
\begin{align}
\label{zonoidinmixedvol3}
&V(K_1[\gamma_1],K_2[\gamma_2],B^n[\beta-\gamma],Z_{1}[\alpha_1],\ldots,Z_m[\alpha_m])
\nonumber\\
&=\binom{n}{\beta,\alpha_1,\ldots,\alpha_m}^{-1}\kappa_{\alpha_1}\cdots \kappa_{\alpha_m}
\int_{\MG(n,\alpha_1)}\ldots \int_{\MG(n,\alpha_m)} [U_1,\ldots,U_m]_{n-\beta}\nonumber\\
&\quad\times
v^{(\beta)}\left(K_1|(U_1+\cdots+U_m)^\bot [\gamma_1],
K_2|(U_1+\cdots+U_m)^\bot [\gamma_2],
B^n|(U_1+\cdots+U_m)^\bot [\beta-\gamma]
\right)\nonumber\\
&\quad\times \varrho_{(\alpha_1)}(Z_1,dU_1)\ldots \varrho_{(\alpha_m)}(Z_m,dU_m).
\end{align}

\begin{proof}[Proof of Theorem \ref{ReverseAlexandrovFenchelKZ}]
Here we have $\gamma_2=0$, and hence $\gamma_1=\gamma\le \beta$, and we set $K_1=:K$. If $\dim(U_1+\cdots +U_m)=n-\beta$, then
\begin{align}\label{boundadd}
&v^{(\beta)}\left(K_1|(U_1+\cdots +U_m)^\bot [\gamma],
B^n|(U_1+\cdots +U_m)^\bot [\beta-\gamma]
\right)\nonumber\\
&=V_\gamma(K|(U_1+\cdots +U_m)^\perp)\frac{\kappa_{\beta-\gamma}}{\binom{\beta}{\gamma}}\nonumber\\
&\le V_\gamma(K) \frac{\kappa_{\beta-\gamma}}{\binom{\beta}{\gamma}}
\end{align}
by Lemma \ref{proj}. Moreover, if $\dim(K)\ge \gamma$, then equality holds if and only if
$K$ is contained in a translate of $(U_1+\cdots +U_m)^\perp$.

Now we insert this estimate into \eqref{zonoidinmixedvol3}. This leads to
\begin{align}
\label{zonoidinmixedvol4}
&V(K[\gamma],B^n[\beta-\gamma],Z_{1}[\alpha_1],\ldots,Z_m[\alpha_m])\nonumber\\
&\le \binom{n}{\gamma,\beta-\gamma,\alpha_1,\ldots,\alpha_m}^{-1}
\kappa_{\beta-\gamma}\kappa_{\alpha_1}\cdots \kappa_{\alpha_m}
V_\gamma(K) \\
&\qquad\times \int_{\MG(n,\alpha_1)}\ldots\int_{\MG(n,\alpha_m)} [U_1,\ldots,U_m]_{n-\beta}\,\varrho_{(\alpha_1)}(Z_1,dU_1)\ldots \varrho_{(\alpha_m)}(Z_m,dU_m)\nonumber\\
&=\binom{\beta}{\gamma}^{-1}\frac{\kappa_{\beta-\gamma}}{\kappa_\beta}V_\gamma(K)V(Z_1[\alpha_1],\ldots,Z_m[\alpha_m],B^n[\beta]).\nonumber
\end{align}
If we use the upper bound $[U_1,\ldots,U_m]_{n-\beta}\le 1$ and \eqref{eq2}, we derive from \eqref{zonoidinmixedvol4} that
\begin{align}
\label{zonoidinmixedvol5}
&V(K[\gamma],B^n[\beta-\gamma],Z_{1}[\alpha_1],\ldots,Z_m[\alpha_m])\nonumber\\
&\le \binom{n }{\gamma, \beta-\gamma,\alpha_1,\ldots, \alpha_m}^{-1}\kappa_{\beta-\gamma}V_\gamma(K)
 V_{\alpha_1}(Z_1)\cdots      V_{\alpha_m}(Z_m).
 \end{align}
This proves \eqref{refeq1.3}.

Now we assume that equality holds. Let $u_0\in\lin(K)$ (the linear subspace parallel to the affine hull of $K$) and $u_i\in\lin(Z_i)$
for $i=1,\ldots,m$. Since $\dim(Z_i)\ge \alpha_i$,
$\lin(Z_i)=\lin\spt(\varrho(Z_i,\cdot))$ and
$$
\lin\{u: u\in U\in\spt(\varrho_{(\alpha_i)}(Z_i,\cdot))\}=\lin(Z_i),
$$
there are $U_i\in \spt(\varrho_{(\alpha_i)}(Z_i,\cdot))$ with $u_i\in U_i$, for $i=1,\ldots,m$. Since in the transition from \eqref{zonoidinmixedvol4} to \eqref{zonoidinmixedvol5} equality holds, we must have $[U_1,\ldots,U_m]_{n-\beta}=1$
if $U_i\in \spt(\varrho_{(\alpha_i)}(Z_i,\cdot))$ for $i=1,\ldots,m$ (note that
$(U_1,\ldots,U_m)\mapsto [U_1,\ldots,U_m]_{n-\beta}$ is continuous), and hence
$U_1,\ldots,U_m$ are pairwise orthogonal. Moreover, equality in
\eqref{zonoidinmixedvol4} implies that we must also have equality in \eqref{boundadd},
hence it follows that $u_0\in
(U_1+\cdots +U_m)^\perp$. Thus we have shown that $u_0,u_1,\ldots,u_m$ are pairwise orthogonal, and hence $K,\lin(Z_1),\ldots,\lin(Z_m)$ lie in pairwise orthogonal affine subspaces.

The reverse implication is clear from the preceding argument, since equality holds in each step.
\end{proof}

\begin{proof}[Proof of Theorem \ref{ReverseAlexandrovFenchelKKZ}] Now we have $\gamma_1=1$, $\gamma_2=\gamma-1\ge 1$, $\beta=\gamma$ and $\alpha_1+\cdots+\alpha_m=n-\gamma=n-\beta$. We apply Theorem \ref{ReverseMinkowski} to the mixed volume in the integrand of \eqref{zonoidinmixedvol3} for $\beta=\gamma$ in
$(U_1+\cdots+U_m)^\perp$ as the ambient space. Here we can assume that $\dim(U_1+\cdots+U_m)=n-\beta$, since otherwise the integrand is zero. This gives
\begin{align}
&v^{(\beta)}\left(K_1|(U_1+\cdots+U_m)^\bot ,
K_2|(U_1+\cdots+U_m)^\bot [\gamma-1]\right)\nonumber\\
&\le\frac{1}{\beta}V_1\left(K_1|(U_1+\cdots+U_m)^\perp\right)V_{\beta-1}\left(K_2|
(U_1+\cdots+U_m)^\perp\right)\label{eqAA}\\
&\le\frac{1}{\beta}V_1(K_1 )V_{\beta-1} (K_2)\label{eqBB} ,
\end{align}
where Lemma \ref{proj} was applied twice in the last step. Since
$\dim(K_1)\ge 1$ and $\dim(K_2)\ge \beta-1$, equality holds in \eqref{eqBB} if and only if $K_1$ and $K_2$ are contained in translates of $(U_1+\ldots+U_m)^\perp$. Then we
deduce from \eqref{zonoidinmixedvol3} that
\begin{align}
&V(K_1,K_2[\gamma-1],Z_1[\alpha_1],\ldots,Z_m[\alpha_m])\nonumber\\
&\le \binom{n}{\beta,\alpha_1,\ldots,\alpha_m}^{-1}\kappa_{\alpha_1}
\cdots\kappa_{\alpha_m}\frac{1}{\beta}V_1(K_1)V_{\beta-1}(K_2)\nonumber\\
&\qquad \times \int_{\MG(n,\alpha_1)}\ldots \int_{\MG(n,\alpha_m)}
[U_1,\ldots,U_m]_{n-\beta}\, \varrho_{(\alpha_1)}(Z_1,dU_1)\ldots \varrho_{(\alpha_m)}(Z_m,dU_m)\nonumber\\
&\le \binom{n}{1,\beta-1,\alpha_1,\ldots,\alpha_m}^{-1}
V_1(K_1)V_{\beta-1}(K_2)V_{\alpha_1}(Z_1)\cdots V_{\alpha_m}(Z_m),\label{eqCC}
\end{align}
where $[U_1,\ldots,U_m]_{n-\beta}\le 1$ and \eqref{eq2} were used in the last step. This proves \eqref{ineq1.5}.

Next we assume that
\begin{align*}
&V(K_1,K_2[\gamma-1],Z_1[\alpha_1],\ldots,Z_m[\alpha_m])\\
&= \binom{n}{1,\beta-1,\alpha_1,\ldots,\alpha_m}^{-1}
V_1(K_1)V_{\beta-1}(K_2)V_{\alpha_1}(Z_1)\cdots V_{\alpha_m}(Z_m).
\end{align*}
Then equality holds in \eqref{eqCC},  \eqref{eqBB} and  \eqref{eqAA}.
Let $v_1\in\lin(K_1)$, $v_2\in\lin(K_2)$ and $u_i\in\lin(Z_i)$ for $i=1,\ldots,m$.
As in the proof of Theorem \ref{ReverseAlexandrovFenchelKZ} we can choose
$U_i\in \spt(\varrho_{(\alpha_i)}(Z_i,\cdot))$ with $u_i\in U_i$ for $i=1,\ldots,m$. Then from equality in \eqref{eqCC} it follows that $[U_1,\ldots,U_m]_{n-\beta}=1$, therefore $U_1,\ldots,U_m$ are pairwise orthogonal and $\dim(U_1+\cdots+U_m)^\perp=\beta$. From \eqref{eqBB} we conclude that $v_1,v_2\in \lin(U_1+\cdots+U_m)^\perp$, and from \eqref{eqAA} we then see that $v_1,v_2$ are orthogonal. This shows that $v_1,v_2,u_1,\ldots,u_m$ are pairwise orthogonal. This finally implies that the affine hulls of $K_1,K_2,Z_1,\ldots,Z_m$ are pairwise orthogonal.

The reverse statement is then also clear from the preceding argument, since equality holds in each step.
\end{proof}

\section{The stability version when all bodies are zonoids}
\label{secstability}

The main goal of this section is a to prove Theorem~\ref{ReverseAlexandrovFenchelZstab}. Although this theorem is stated for zonoids, most of the
preparatory statements are valid for arbitrary convex bodies.

For $m=1,\ldots,n$, the $m$-dimensional size of a compact convex set $K$ can be effectively measured by the
largest radius $r_m(K)$ of $m$-dimensional balls contained in $K$.

\begin{lemma}
\label{rmellipsoid}
Let $1\leq m\leq d\leq n$, and let $E$ be a $d$-dimensional ellipsoid in $\R^n$ with
half-axes $a_1\geq\ldots\geq a_d>0$.
\begin{enumerate}
\item[{\rm (i)}] $r_{m}(E)=a_m$.
\item[{\rm (ii)}] There exists some $A\in\MA(n,m-1)$ such that
$E\subset A+r_m(E) B^n$ and even $r_{m-1}(E|A)=r_{m-1}(E)$ provided $m\geq 2$. If $o\in E$, then $A $  can be chosen as a linear subspace.
\item[{\rm (iii)}] For any $L\in \MG(n,m)$, we have $r_{m}(E|L)\leq r_{m}(E)$.
\end{enumerate}
\end{lemma}

\begin{proof}
We may assume that the origin is the center of $E$ and $d=n$. Let
$e_1,\ldots,e_n$ be an orthonormal basis of $\R^n$ such that
\begin{equation}
\label{Eequation}
E=\left\{\sum_{i=1}^nt_ie_i:\,\sum_{i=1}^n\frac{t_i^2}{a_i^2}\leq 1\right\}.
\end{equation}
For (i), there is  an $m$-ball $B$ of radius $r_m(E)$ contained in $E$ and centered at the origin $o$ because $E$ is $o$-symmetric.
It follows that there exists a $y\in B\cap \lin\{e_m,\ldots,e_n\}$ with $\|y\|=r_m(E)$.
Since $\|y\|\leq a_m$ by \eqref{Eequation}, we conclude (i).

For (ii), we can take $A=\{o\}$ if $m=1$, and $A=\lin\{e_1,\ldots,e_{m-1}\}$ if $m\geq 2$ by (i) and \eqref{Eequation}.

For (iii), let $L\in \MG(n,m)$. Since $E$ is a linear image of $B^n$, there exists an $L_0\in\MG(n,m)$ such that
$E|L=(E\cap L_0)|L$; therefore, $r_{m}(E|L)\leq r_m(E\cap L_0)\leq r_{m}(E)$.
\end{proof}

According to John's theorem (see Schneider \cite[Theorem 10.12.2]{Sch14}), for any $d$-dimensional compact convex set $K$ in $\R^n$, $1\leq d\leq n$,
there exists a $z\in K$ and a $d$-dimensional ellipsoid $E$ centred at $o$ such that
\begin{equation}
\label{John}
z+E\subset K\subset z+d\cdot E.
\end{equation}

\begin{lemma}
\label{rm+1A}
Let $K$ be a compact convex set in $\R^n$, and let $m\in\{1,\ldots,n-1\}$.
\begin{enumerate}
\item[{\rm (i)}] If $A\in \MA(n,m)$  and $K\subset A+\varrho B^n$
for some $\varrho\geq 0$, then $r_{m+1}(K)\leq \varrho$.
\item[{\rm (ii)}] There exists some $A_0\in \MA(n,m)$ such that
$K\subset A_0+nr_{m+1}(K) B^n$ and $r_m(K|A_0)\geq \frac1n\, r_m(K)$.  where
$o\in A_0$ If $K$ is $o$-symmetric, then $A_0$  can be chosen as a linear subspace.
\item[{\rm (iii)}] If $L\in \MG(n,m)$, then $r_m(K|L)\leq n r_m(K)$.
\end{enumerate}
\end{lemma}

\begin{proof}
For (i), let $B$ be an $(m+1)$-dimensional ball of radius $r_{m+1}(K)$ and center $z$ contained in $K$.
It follows that $(z+A^\bot)\cap B$ contains a segment $s$ of length  $2r_{m+1}(K)$.
Since $s$ is a subset of the $(n-m)$-dimensional ball $(z+A^\bot)\cap (A+\varrho B^n)$ of radius $\varrho$,
we have $r_{m+1}(K)\leq \varrho$.

For (ii), let $z\in K$ and let $E$ be an ellipsoid centred at $o$ with ${\rm dim}\,E={\rm dim}\,K$ and satisfying
\eqref{John}. It follows from Lemma~\ref{rmellipsoid} (ii) that
there exists $L_0\in \MG(n,m)$ such that
$E\subset L_0+r_{m+1}(E) B^n$ and $r_m(E|L_0)=r_m(E)$.
Since $r_{m+1}(E)\leq r_{m+1}(K)$ and $E\subset K-z\subset nE$, we conclude that with $A_0:=z+L_0$ we have
$K\subset A_0+nr_{m+1}(K) B^n$ and $r_m(K|A_0)\geq r_m(E|L_0)=r_m(E)\geq\frac1n\,r_m(K)$.

For the proof of (iii), we choose $z,E$ as for (ii). Then it follows from  Lemma~\ref{rmellipsoid} (iii) that
$$
r_m(K|L)\leq r_m((z+nE)|L)\leq r_m(z+nE)= n\,r_m(z+E)\leq n r_m(K),
$$
which yields the assertion.
\end{proof}

The main tool used to verify stability versions of
Theorems~\ref{ReverseAlexandrovFenchelKZ} and \ref{ReverseAlexandrovFenchelKKZ} is
Proposition~\ref{projstab}.

\begin{prop}
\label{projstab}
Let $n\geq 2$ and $1\leq \beta\leq n-1$.
If  $K$ is a compact convex set in $\R^n$ with $\dim(K)\ge \beta$, then
$$
V_\beta(K)\geq \left(1+\frac1{2^{n+2}n^5}\cdot\frac{r_{\beta+1}(K)^2}{r_\beta(K)^2}\right)\max\{
V_\beta(K|L):L\in \MG(n,\beta)\}.
$$
\end{prop}

\begin{remark}{\rm  $(\beta+1)$-dimensional cones show that the order of the estimate is optimal.}
\end{remark}

\begin{proof}
We may assume that $r_{\beta+1}(K)>0$ (compare Lemma~\ref{proj}) and that the origin $o$ is the center of mass of $K$.

Let $L_\beta \in \MG(n,\beta)$ be such that $V_\beta(K|L_\beta)=\max\{V_\beta(K|L):L\in \MG(n,\beta)\}$, and
let $B$ be a $(\beta+1)$-dimensional ball of radius $r_{\beta+1}(K)$ contained in $K$. Writing $z$ to denote the center of $B$,
it follows that there exists a segment $s\subset B\cap (z+L_\beta^\bot)$ of length $2r_{\beta+1}(K)$, where $s=[z-r_{\beta+1}(K)v,z+r_{\beta+1}(K)v]$
for a suitable $v\in L_\beta^\bot\cap \BS^{n-1}$. On the other hand, according to Lemma~\ref{rm+1A}, there exists a $(\beta-1)$-dimensional affine subspace
$A_{\beta-1}\subset L_\beta$ (note that this is obviously true if $\beta=1$) such that
\begin{equation}
\label{K|LbetaAbeta1}
K|L_\beta\subset A_{\beta-1}+nr_\beta(K|L_\beta)B^n\subset A_{\beta-1}+n^2r_\beta(K)B^n.
\end{equation}
Let $L_{\beta-1}\subset L_\beta$ be the $(\beta-1)$-dimensional linear subspace parallel to $A_{\beta-1}$.

We consider the $(\beta+1)$-dimensional linear subspace
$L_{\beta+1}={\rm lin}\{L_\beta,v\}$ and an orthonormal basis $u_1,\ldots,u_n$ of $\R^n$, where
$u_1,\ldots,u_\beta\in L_\beta$,  $L_{\beta-1}=u_\beta^\bot\cap L_\beta$
and $u_{\beta+1}=v$.
For $i=0,\ldots,n$, we define $K_i$ by reverse induction; namely, $K_n=K$ and $K_{i-1}=S_{u_i^\bot}K_i$
for $i=1,\ldots,n$. We deduce from \eqref{steiner} and reverse induction on $i$ that
\begin{equation}
\label{VbetaKi}
V_\beta(K_i)\leq V_\beta(K)\mbox{ \ for $i=0,\ldots,n$.}
\end{equation}
We observe that
$$
K|L_{\beta+1}=K_{\beta+1}|L_{\beta+1}\subset K_{\beta+1}
$$
and
$$
[(z|L_{\beta+1})-r_{\beta+1}(K)v,(z|L_{\beta+1})+r_{\beta+1}(K)v]\subset K_{\beta+1}\cap L_{\beta+1},
$$
which in turn yields
$$
K|L_\beta=K_{\beta}|L_\beta\subset K_{\beta}
\mbox{ \ and \ }
[(z|L_\beta)-r_{\beta+1}(K)v,(z|L_\beta)+r_{\beta+1}(K)v]\subset K_{\beta}\cap L_{\beta+1}.
$$
By construction, $K_0$ is unconditional with respect to the basis $u_1,\ldots,u_n$. Moreover, \eqref{K|LbetaAbeta1} and an application of \eqref{steinerequa} in $L_\beta$ show that
\begin{eqnarray}
\label{K0Vbeta0}
V_\beta(K_0\cap L_\beta)&=&V_\beta(K|L_\beta),\\
\label{K0Lbetathin0}
K_0\cap L_\beta&\subset &L_{\beta-1}+n^2r_\beta(K)B^n,\\
\label{K0v0}
[-r_{\beta+1}(K)v,r_{\beta+1}(K)v]&\subset& K_0.
\end{eqnarray}

We now transform $K_0$ into a convex body whose sections parallel to $L_{\beta-1}$ are $(\beta-1)$-dimensional balls, by possibly further symmetrizing $K_0$.
If $\beta=1$, then let $K_0^{(\beta-1)}=K_0$. If $\beta\geq 2$, then
by applying repeated Steiner symmetrizations with respect to suitable subspaces
$u^\bot$ with $u\in L_{\beta-1}$ and taking limits, we obtain a compact convex set $K_0^{(\beta-1)}$ (see \cite{Gru07} or \cite[\S 19]{Leicht}) such that
\begin{eqnarray*}
K_0^{(\beta-1)}|L_{\beta-1}^\bot&=&K_0|L_{\beta-1}^\bot,\\
K_0^{(\beta-1)}\cap (x+L_{\beta-1})&{\rm is} &\mbox{a $(\beta-1)$-ball centered at $x$ if $x\in(\text{\rm relint }K_0)|L_{\beta-1}^\bot$},\\
V_{\beta-1}(K_0^{(\beta-1)}\cap (x+L_{\beta-1}))&=&V_{\beta-1}(K_0\cap (x+L_{\beta-1}))\mbox{ if $x\in(\text{\rm relint }K_0)|L_{\beta-1}^\bot$}.
\end{eqnarray*}
From the properties \eqref{steiner} and \eqref{steinerequa} of the Steiner symmetrization, and from
\eqref{VbetaKi}, \eqref{K0Vbeta0}, \eqref{K0Lbetathin0} and \eqref{K0v0} we deduce that
\begin{eqnarray}
\label{K0betaVbeta}
V_\beta(K_0^{(\beta-1)})&\leq &V_\beta(K),\\
\label{K0Vbeta}
V_\beta(K_0^{(\beta-1)}\cap L_\beta)&=&V_\beta(K|L_\beta),\\
\label{K0Lbetathin}
K_0^{(\beta-1)}\cap L_\beta&\subset &L_{\beta-1}+n^2r_\beta(K)B^n,\\
\label{K0v}
[-r_{\beta+1}(K)v,r_{\beta+1}(K)v]&\subset& K_0^{(\beta-1)}.
\end{eqnarray}
It follows from \eqref{K0betaVbeta}, \eqref{K0Vbeta} and \eqref{K0v} that  the unconditional $(\beta+1)$-dimensional compact convex set
\begin{equation}
\label{tildeKdef}
\widetilde{K}=[K_0^{(\beta-1)}\cap L_\beta,-r_{\beta+1}(K)v,r_{\beta+1}(K)v]
\end{equation}
satisfies
\begin{equation}
\label{KtildeK}
V_\beta(\widetilde{K})\leq V_\beta(K)\mbox{ \ and \ }
V_\beta(\widetilde{K}\cap L_\beta)=V_\beta(K|L_\beta).
\end{equation}
We define the set
\begin{equation}
\label{tildeK0def}
\widetilde{K}_0=\widetilde{K}\cap L_\beta=K_0^{(\beta-1)}\cap L_\beta,
\end{equation}
which has axial rotational symmetry around $\R u_\beta$ and
 according to \eqref{K0Lbetathin} satisfies
\begin{equation}
\label{tildeK0thin}
\widetilde{K}_0\subset L_{\beta-1}+n^2r_\beta(K)B^n.
\end{equation}
According to \eqref{KtildeK}, Proposition~\ref{projstab} will follow once we have shown that
\begin{equation}
\label{projstabtildeK}
V_\beta(\widetilde{K})\geq \left(1+\frac1{2^{n+2}n^5}\cdot \frac{r_{\beta+1}(K)^2}{r_\beta(K)^2}\right)V_\beta(\widetilde{K}_0).
\end{equation}
For the proof, we will use \eqref{tildeKdef}, \eqref{tildeK0def} and \eqref{tildeK0thin}. We distinguish the cases $\beta=1$ and $\beta\geq 2$.

\medskip

\noindent {\bf Case 1:} $\beta=1$.
In this case, $\widetilde{K}$ is a rhombus, one of the diagonals is $\widetilde{K}_0=\widetilde{K}\cap L_1=K_0\cap L_1$, which has length
$$
V_1(\widetilde{K}_0)=V_1(K_0\cap L_1)=V_1(K|L_1)\le 2r_1(K|L_1)\le 2nr_1(K),
$$
and the other one is of length
$2r_2(K)$. Hence it follows that
\begin{eqnarray*}
V_1(\widetilde{K})&=&\frac{1}{2}\cdot 4\cdot \sqrt{\left(\frac{V_1(\widetilde{K}_0)}{2}\right)^2+r_2(K)^2}=
V_1(\widetilde{K}_0)\sqrt{1+\frac{4r_2(K)^2}{V_1(\widetilde{K}_0)^2}}\\
&\geq&V_1(\widetilde{K}_0)\sqrt{1+\frac{r_2(K)^2}{n^2r_1(K)^2}}\geq
V_1(\widetilde{K}_0)\left(1+\frac{r_2(K)^2}{4n^2r_1(K)^2}\right).
\end{eqnarray*}

\noindent {\bf Case 2:} $\beta\geq 2$.
We write $\partial \widetilde{K}$ to denote the relative boundary of $\widetilde{K}$ in $L_{\beta+1}$,
$\partial \widetilde{K}_0$ to denote the relative boundary of $\widetilde{K}_0$ in $L_{\beta}$
and
define $f:\,\widetilde{K}_0\to \R_{\geq 0}$  to be the concave function satisfying
$x+f(x)v\in\partial \widetilde{K}$ for $x\in \widetilde{K}_0$. In particular, we have
\begin{equation}
\label{ftildeK}
V_\beta(\widetilde{K})=\int_{\widetilde{K}_0}\sqrt{1+\|\nabla f(x)\|^2} \,dx,
\end{equation}
where the integration is with respect to the $\beta$-dimensional Lebesgue measure on $L_\beta$, and the derivative $\nabla f$ of $f$ is well-defined for almost all $x\in \widetilde{K}_0$.

Let $x\in\text{\rm relint } \widetilde{K}_0$ be such that $\nabla f(x)$ exists. Hence  $x\neq o$ and there exist
$z\in \partial\widetilde{K}_0$ and $t\in(0,1)$ such that $x=tz$. Since $f$ is concave and
$\widetilde{K}$ is a double cone, we deduce that
\begin{eqnarray*}
f(x)+\langle \nabla f(x),y-x\rangle&\geq& f(y) \mbox{ \ for  $y\in \widetilde{K}_0$},\\
f(x)+\langle \nabla f(x),-x\rangle&=&f(o),\\
f(x)+\langle \nabla f(x),z-x\rangle&=&f(z)=0.
\end{eqnarray*}
Using also \eqref{tildeKdef}, we deduce that if $x=tz$  for $z\in \partial\widetilde{K}_0$ and $t\in(0,1)$, and $\nabla f(x)$ exists, then
\begin{eqnarray}
\label{ftildeKo}
f(o)&= & r_{\beta+1}(K),\\
\label{ftildeKy}
f(o)+\langle \nabla f(x),y\rangle\geq f(y)&\geq&0\mbox{ \ for $y\in \widetilde{K}_0$},\\
\label{ftildeKz}
f(o)+\langle \nabla f(x),z\rangle= f(z)&=&0.
\end{eqnarray}

Next we define
\begin{eqnarray*}
\Xi&=&\{z\in \partial\widetilde{K}_0:\,z|L_{\beta-1}\in \mbox{$\frac12$}\,\widetilde{K}_0\}\\
\widetilde{\Xi}&=&\{tz:\,z\in\Xi\mbox{ and }t\in[0,1]\}.
\end{eqnarray*}
We claim that if $x\in\widetilde{\Xi}$ and  $\nabla f(x)$ exists, then
\begin{equation}
\label{nablaflowerest}
\sqrt{1+\|\nabla f(x)\|^2}\geq 1+\frac{r_{\beta+1}(K)^2}{16n^4r_\beta(K)^2}.
\end{equation}
For the proof, we can write $x$ as
 $x=tz$ for $z\in\Xi$  and $t\in(0,1]$ and we have  $z_0=z|L_{\beta-1}\in\frac12\,\widetilde{K}_0$. There exists an $s\in\R$
such that $z-z_0=su_\beta$, and hence $y=2z_0\in \widetilde{K}_0$ satisfies
\begin{equation}
\label{zys}
 2su_\beta=2z-y.
\end{equation}
We deduce from \eqref{ftildeKy}, \eqref{ftildeKz} and \eqref{zys} that
\begin{equation}
\label{fos}
f(o)+\langle \nabla f(x),2su_\beta\rangle=2\left[f(o)+\langle\nabla f(x),z\rangle\right]-\left[f(o)+\langle \nabla f(x),y\rangle\right]\leq 0.
\end{equation}
Here $0<|s|\leq n^2r_\beta(K)$ by \eqref{tildeK0thin} and \eqref{fos}. Therefore \eqref{ftildeKo} and \eqref{fos} imply that
$$
\|\nabla f(x)\|\geq\frac{f(o)}{2|s|}\geq \frac{r_{\beta+1}(K)}{2n^2r_\beta(K)}.
$$
Using $\sqrt{1+t}\geq 1+\frac14\,t$ for $t\in[0,1]$, we thus obtain \eqref{nablaflowerest}. Let $\lambda_\beta$ be the Lebesgue-measure in $L_\beta$. Combining \eqref{ftildeK} and \eqref{nablaflowerest}, we get
\begin{align}
\label{VbetatildeKest}
V_\beta(\widetilde{K})&=\int_{\widetilde{K}_0}\sqrt{1+\|\nabla f(x)\|^2}\,dx\nonumber\\
&=\int_{\widetilde{K}_0\setminus\widetilde{\Xi}}\sqrt{1+\|\nabla f(x)\|^2}\,dx+\int_{\widetilde{\Xi}}\sqrt{1+\|\nabla f(x)\|^2}\,dx\nonumber\\
&\ge \lambda_\beta(\widetilde{K}_0\setminus\widetilde{\Xi})+
\left(1+\frac{1}{4} \left(\frac{r_{\beta+1}(K)}{2n^2r_\beta(K)}\right)^2\right)
\lambda_\beta(\widetilde{\Xi})\nonumber\\
&= V_\beta(\widetilde{K}_0)+
\frac{r_{\beta+1}(K)^2}{16n^4r_\beta(K)^2}\cdot \lambda_\beta(\widetilde{\Xi}).
\end{align}

What is left to find  is a lower bound for the ratio $\lambda_\beta(\widetilde{\Xi})/V_\beta(\widetilde{K}_0)$. As $\widetilde{K}_0$ is unconditional within $L_\beta$, we have
$$
{\rm conv}\,\widetilde{\Xi}=\{x\in \widetilde{K}_0:\,x|L_{\beta-1}\in \mbox{$\frac12$}\,\widetilde{K}_0\},
$$
which   is again an unconditional set. We observe that the linear transformation $\Phi:\,L_\beta\to L_\beta$
with $\Phi u_\beta=u_\beta$ and $\Phi u_i=\frac12\,u_i$ for $i=1,\ldots,\beta-1$ satisfies
$\Phi \widetilde{K}_0\subset {\rm conv}\,\widetilde{\Xi}$, therefore
\begin{equation}
\label{halftildeK0}
V_\beta({\rm conv}\,\widetilde{\Xi})\geq 2^{-(\beta-1)}V_\beta(\widetilde{K}_0).
\end{equation}
As $\widetilde{K}_0$ has axial rotational symmetry around $\R u_\beta$, the same holds for $\widetilde{\Xi}$. In particular,
there exist $\varrho>0$ and $s_0>0$ such that
\begin{eqnarray*}
L_{\beta-1}\cap {\rm conv}\,\widetilde{\Xi}&= & L_{\beta-1}\cap \varrho B^n,\\
Z &\subset&{\rm conv}\,\widetilde{\Xi}\mbox{ \ for the cylinder $(L_{\beta-1}\cap \varrho B^n)+[-s_0u_\beta,s_0u_\beta]$,}\\
({\rm conv}\,\widetilde{\Xi})\setminus Z&\subset&\widetilde{\Xi}.
\end{eqnarray*}
Since $\widetilde{\Xi}\cap Z$ is the union of two cones with height $s_0$, using \eqref{halftildeK0} we deduce  that
$$
\frac{\lambda_\beta(\widetilde{\Xi})}{V_\beta(\widetilde{K}_0)}\geq
\frac{\lambda_\beta(\widetilde{\Xi})}{2^{\beta-1}V_\beta({\rm conv}\,\widetilde{\Xi})}
=\frac{1}{2^{\beta-1}}\frac{\lambda_\beta(\widetilde{\Xi}\setminus Z)+
\lambda_\beta(\widetilde{\Xi}\cap Z)}{V_\beta(\widetilde{\Xi}\setminus Z)+
V_\beta(  Z)}
\geq
\frac{\lambda_\beta(\widetilde{\Xi}\cap Z)}{2^{\beta-1}V_\beta(Z)}=\frac1{\beta\,2^{\beta-1}}.
$$
It follows from \eqref{VbetatildeKest} and $\beta\leq n-1$ that
$$
V_\beta(\widetilde{K})\geq\left(1+\frac1{\beta\,2^{\beta-1}16n^4}\cdot
\frac{r_{\beta+1}(K)^2}{r_\beta(K)^2} \right) V_\beta(\widetilde{K}_0)
\geq\left(1+\frac1{2^{n+2}n^5}\cdot
\frac{r_{\beta+1}(K)^2}{r_\beta(K)^2} \right) V_\beta(\widetilde{K}_0).
$$
We conclude \eqref{projstabtildeK}, and in turn Proposition~\ref{projstab}.
\end{proof}

\medskip

\begin{prop}
\label{oneKstab}
Let $K_1$ be a compact convex set in $\R^n$,  let $Z_2,\ldots,Z_{m}$, $2\leq m\leq n$, be zonoids  in $\R^n$, and let $\alpha_1,\ldots,\alpha_m\in\N$ be such that
$\alpha_1+\cdots+\alpha_m=n$. Let $\varepsilon\in [0,1]$. Suppose that
\begin{align}
0&<V_{\alpha_1}(K_1)\prod_{i=2}^m V_{\alpha_i}(Z_i)
  \leq (1+\varepsilon)
\binom{n}{\alpha_1,\ldots, \alpha_m}
V(K_{1}[\alpha_1],Z_2[\alpha_2],\ldots,Z_{m}[\alpha_{m}]).\label{hypo1}
\end{align}
Then there exists an affine subspace $A_1\in\MA(n,\alpha_1)$ such that
$$
K_{1}\subset A_1+c\sqrt{\varepsilon}\, r_{\alpha_1}(K_1) B^n \quad
\mbox{ and } \quad
r_{\alpha_1}(K_1|A_1)\geq \frac{r_{\alpha_1}(K_1)}n,
$$
where $c=2^{\frac{n+2}2}n^{\frac72}$, and $A_1$ can be chosen as a linear subspace if $K_1$ is a zonoid.
\end{prop}

\begin{proof}
Let $L_1\in \MG(n,\alpha_1)$ be such that
$$
V_{\alpha_1}(K_1|L_1)=\max\left\{V_{\alpha_1}(K_1|L) : L\in \MG(n,\alpha_1)\right\}.
$$
A special case of \eqref{zonoidinmixedvol3}, \eqref{Number1} and \eqref{eq2} imply that
$$
\binom{n}{ \alpha_1,\ldots, \alpha_m}
V(K_{1}[\alpha_1],Z_2[\alpha_2],\ldots,Z_{m}[\alpha_{m}])
\leq
V_{\alpha_1}(K_1|L_1)\cdot
 V_{\alpha_2}(Z_2)\cdots V_{\alpha_{m}}(Z_{\alpha_m}).
$$
It follows from the assumption \eqref{hypo1} in Proposition~\ref{oneKstab} that
$V_{\alpha_1}(K_1)\leq(1+\varepsilon)V_{\alpha_1}(K_1|L_1)$, and hence
Proposition~\ref{projstab}  yields
$$
r_{\alpha_1+1}(K_1)\leq 2^{\frac{n+2}2}n^{\frac52}r_{\alpha_1}(K_1)\sqrt{\varepsilon}.
$$
From Lemma~\ref{rm+1A} we deduce the existence of an  affine subspace $A_1\in\MA(n,\alpha_1)$ such that
$$
K_1\subset A_1+2^{\frac{n+2}2}n^{\frac72}r_{\alpha_1}(K_1)\sqrt{\varepsilon} \,B^n \quad
\mbox{ and } \quad
r_{\alpha_1}(K_1|A_1)\geq \frac{r_{\alpha_1}(K_1)}n,
$$
where we can choose $A_1\in\MG(n,\alpha_1)$ if $K_1$ is a zonoid.
\end{proof}

\begin{lemma}
\label{anglesubspaces}
Let $2\leq m\leq n$, let $K_1,\ldots,K_m$ be compact, convex sets in $\R^n$,  and let $\alpha_1,\ldots,\alpha_m\geq 1$ be integers satisfying
$\alpha_1+\cdots+\alpha_m=n$. Let $\varepsilon \in [0,1]$. Assume that
 for $i=1,\ldots,m$ there exist $q_i\in K_i$ and linear   subspace
$L_i\in\MG(n,\alpha_i)$ such that $K_i|L_i$ is a zonoid for $i=2,\ldots,m$,
\begin{equation}
\label{KiLinear}
K_i-q_i\subset (K_i|L_i)+\varepsilon r_{\alpha_i}(K_i) \,B^n\quad
\mbox{ and } \quad
r_{\alpha_i}(K_i|L_i)\geq \frac{r_{\alpha_i}(K_i)}n>0,
\end{equation}
and $K_1,\ldots,K_m$ satisfy
\begin{align}
\label{KiLicond}
&\binom{n}{ \alpha_1,\ldots, \alpha_m}
V(K_1[\alpha_1],\ldots,K_m[\alpha_m])\nonumber\\
&\qquad\qquad \leq\prod_{i=1}^mV_{\alpha_i}(K_i)\leq (1+\varepsilon)
\binom{n}{ \alpha_1,\ldots, \alpha_m}
V(K_1[\alpha_1],\ldots,K_m[\alpha_m]).
\end{align}
Then
$$
[L_1,\ldots,L_m]\geq 1-n^{5} \varepsilon.
$$
\end{lemma}

\begin{proof} If $\varepsilon\ge n^{-4}$, then $1-n^5\varepsilon\le -1$ and hence there is nothing to show. In the following, we may assume that $\varepsilon< n^{-4}$.

By assumption, the sets $C_i=K_i|L_i$ are zonoids for $i=2,\ldots,m$. Using first \eqref{KiLicond} and then \eqref{KiLinear} we get
\begin{align}
\nonumber
\prod_{i=1}^mV_{\alpha_i}(K_i)
&\leq (1+\varepsilon)
\binom{n }{ \alpha_1,\ldots, \alpha_m}
V(K_1[\alpha_1],\ldots,K_m[\alpha_m])\\
\nonumber
&\leq \binom{n }{ \alpha_1,\ldots, \alpha_m}
V(K_1[\alpha_1],\ldots,K_m[\alpha_m]) +\varepsilon\cdot \prod_{i=1}^mV_{\alpha_i}(K_i)\\
\nonumber
&\leq \binom{n }{ \alpha_1,\ldots, \alpha_m}
V((C_1+\varepsilon\, r_{\alpha_1}(K_1) \,B^n)[\alpha_1],\ldots,
(C_m+\varepsilon\,r_{\alpha_m}(K_m) \,B^n)[\alpha_m])\\
\nonumber
&\quad  +\varepsilon\cdot \prod_{i=1}^mV_{\alpha_i}(K_i)\\
\nonumber
&= \binom{n }{ \alpha_1,\ldots, \alpha_m} V(C_1[\alpha_1],\ldots,C_m[\alpha_m])+
 \Omega+\varepsilon\cdot \prod_{i=1}^mV_{\alpha_i}(K_i)\\
\label{Omega-appearK}
&= \left(\prod_{i=1}^mV_{\alpha_i}(C_i)\right)[L_1,\ldots,L_m]+
 \Omega+\varepsilon\cdot \prod_{i=1}^mV_{\alpha_i}(K_i),
\end{align}
where with
$$
\Gamma=\{(\gamma_1,\ldots,\gamma_m):\,\gamma_i\in\{0,\ldots,\alpha_i\}\mbox{ for }i=1,\ldots,m \mbox{ and }
\sum_{i=1}^m\gamma_i\geq 1\}
$$
and by the linearity of  mixed volumes we have
\begin{align}
\Omega&=\binom{n }{ \alpha_1,\ldots, \alpha_m}\big[ V((C_1+\varepsilon\,r_{\alpha_1}(K_1) \,B^n)[\alpha_1],\ldots,
(C_m+\varepsilon\,r_{\alpha_m}(K_m)\,B^n)[\alpha_m])\nonumber\\
&\qquad\qquad\qquad\qquad  -V(C_1[\alpha_1],\ldots,C_m[\alpha_m])\big]\nonumber\\
&=\binom{n }{ \alpha_1,\ldots, \alpha_m} \sum_{(\gamma_1,\ldots,\gamma_m)\in\Gamma}
\varepsilon^{\sum_{i=1}^m\gamma_i}\left(\prod_{i=1}^m r_{\alpha_i}(K_i)^{\gamma_i}\right)
\prod_{i=1}^m \binom{\alpha_i}{ \gamma_i}\nonumber\\
\label{Omega-formulaK}
&\qquad \times V\left(C_1[\alpha_1-\gamma_1],\ldots,C_m[\alpha_m-\gamma_m],B^n
\left[ \gamma_1+\cdots+\gamma_m\right]\right).
\end{align}
For each $(\gamma_1,\ldots,\gamma_m)\in\Gamma$, it follows from Theorem~\ref{ReverseAlexandrovFenchelKZ} (note that $K$ in the statement of  Theorem~\ref{ReverseAlexandrovFenchelKZ} need not be a zonoid, hence here we do not have to assume that $C_1$ is a zonoid) that
\begin{align}\label{Omega-part1K}
\nonumber
 &V\left(C_1[\alpha_1-\gamma_1],\ldots,C_m[\alpha_m-\gamma_m],B^n
 \left[ \gamma_1+\cdots+\gamma_m\right]\right)\nonumber\\
 &\qquad \leq
\binom{n }{  \alpha_1-\gamma_1,\ldots, \alpha_m-\gamma_m,\gamma_1+\cdots+\gamma_m}^{-1}\kappa_{\gamma_1+\cdots+\gamma_m}
 \left(\prod_{i=1}^m V_{\alpha_i-\gamma_i}(C_i)\right).
\end{align}
Using the notation $B_i^{\alpha_i}=B^n\cap L_i$ for $i=1,\ldots,m$, it follows from \eqref{KiLinear} that
$p_i+\frac{r_{\alpha_i}(K_i)}{n}\,B_i^{\alpha_i}\subset C_i$ for some $p_i\in C_i$. Therefore,
using Lemma~\ref{proj} at the end, we have
\begin{align}
\nonumber
r_{\alpha_i}(K_i)^{\gamma_i}V_{\alpha_i-\gamma_i}(C_i)&=
\frac{\binom{\alpha_i}{ \gamma_i}}{\kappa_{\gamma_i}}\cdot r_{\alpha_i}(K_i)^{\gamma_i}v^{(\alpha_i)}(C_i[\alpha_i-\gamma_i],
B_i^{\alpha_i}[\gamma_i])\\
\nonumber
&= \frac{n^{\gamma_i}\binom{\alpha_i}{ \gamma_i}}{\kappa_{\gamma_i}}\cdot
 v^{(\alpha_i)}\left(C_i[\alpha_i-\gamma_i],\frac{r_{\alpha_i}(K_i)}{n}\,
 B_i^{\alpha_i}[ \gamma_i]\right))\\
\nonumber
&\leq  \frac{n^{\gamma_i}\binom{\alpha_i}{ \gamma_i}}{\kappa_{\gamma_i}}\cdot
v^{(\alpha_i)}\left(C_i[\alpha_i-\gamma_i],C_i[ \gamma_i]\right)\\
\label{Omega-part2K}
&=\frac{n^{\gamma_i}\binom{\alpha_i}{ \gamma_i}}{\kappa_{\gamma_i}}\cdot V_{\alpha_i}(C_i)\leq
\frac{n^{\gamma_i}\binom{\alpha_i}{ \gamma_i}}{\kappa_{\gamma_i}}\cdot V_{\alpha_i}(K_i).
\end{align}
In the following, we use the rough upper bounds
$$
\binom{\sum_{i=1}^m\gamma_i }{ \gamma_1,\ldots ,\gamma_m}\le n^{\sum_{i=1}^m\gamma_i},\qquad
\frac{\kappa_{\sum_{i=1}^m\gamma_i}}{\prod_{i=1}^m\kappa_{ \gamma_i}}\le 1
$$
and  $t\leq e^t-1\leq 2t$, where the right inequality holds for $t\in[0,1)$.
Combining these facts with \eqref{Omega-formulaK}, \eqref{Omega-part1K}, \eqref{Omega-part2K},
we thus obtain for $n\ge 3$ that
\begin{align}\label{Omega-est1K}
\nonumber
\Omega&\le
\prod_{i=1}^m V_{\alpha_i}(K_i)\cdot
 \sum_{(\gamma_1,\ldots,\gamma_m)\in\Gamma}
 \frac{\binom{n }{ \alpha_1,\ldots, \alpha_m}}{\binom{n}{\alpha_1-\gamma_1,\ldots,\alpha_m-\gamma_m,\gamma_1+\cdots+\gamma_m}}
 \frac{\kappa_{\sum_{i=1}^m\gamma_i}}{\prod_{i=1}^m\kappa_{\gamma_i}}
 n^{\sum_{i=1}^m\gamma_i}\prod_{i=1}^m \binom{\alpha_i}{ \gamma_i}^2
 \varepsilon^{\sum_{i=1}^m\gamma_i}\nonumber\\
&\leq  \prod_{i=1}^m V_{\alpha_i}(K_i) \cdot\sum_{(\gamma_1,\ldots,\gamma_m)\in\Gamma}
\binom{\gamma_1+\cdots+\gamma_m}{\gamma_1,\ldots,\gamma_m}
n^{\sum_{i=1}^m\gamma_i}\varepsilon^{\sum_{i=1}^m\gamma_i}
\prod_{i=1}^m \binom{\alpha_i}{ \gamma_i}\nonumber\\
&\leq  \prod_{i=1}^m V_{\alpha_i}(K_i)\cdot \sum_{(\gamma_1,\ldots,\gamma_m)\in\Gamma}
\left(n^2\varepsilon\right)^{\sum_{i=1}^m\gamma_i}\prod_{i=1}^m \binom{\alpha_i}{ \gamma_i}\nonumber\\
&=\prod_{i=1}^m V_{\alpha_i}(K_i)\cdot \left[\prod_{i=1}^m\left(1+n^2\varepsilon\right)^{\alpha_i}-1\right]\nonumber\\
&\le \prod_{i=1}^m V_{\alpha_i}(K_i)\cdot \left[\exp\left(n^2\varepsilon (\alpha_1+\cdots+\alpha_m)\right)-1\right]
=\prod_{i=1}^m V_{\alpha_i}(K_i)\cdot \left(\exp\left(n^3\varepsilon\right)-1\right)\nonumber\\
&\le \prod_{i=1}^m V_{\alpha_i}(K_i)\cdot 2 n^3\varepsilon\le (n^5-1)\prod_{i=1}^m V_{\alpha_i}(K_i)\cdot \varepsilon,
\end{align}
where $n\ge 2$ was used in the last step.

Substituting \eqref{Omega-est1K} into \eqref{Omega-appearK},  for $ \varepsilon <n^{-4}$ we get
$$
\prod_{i=1}^mV_{\alpha_i}(K_i)\leq \left(\prod_{i=1}^mV_{\alpha_i}(C_i)\right)[L_1,\ldots,L_m]+
n^{5}\left(\prod_{i=1}^m V_{\alpha_i}(K_i)\right)\cdot  \varepsilon.
$$
Since $V_{\alpha_i}(C_i)\leq V_{\alpha_i}(K_i)$, according to Lemma~\ref{proj}, and $\dim(K_i)\ge \alpha_i$,   $i=1,\ldots,m$, we deduce that
$$
1-n^{5}\varepsilon\leq [L_1,\ldots,L_m],
$$
which completes the argument.
\end{proof}

\bigskip

\begin{proof}[Proof of Theorem~\ref{ReverseAlexandrovFenchelZstab}]
We deduce from the assumptions and by repeated application of  Proposition~\ref{oneKstab} that for $i=1,\ldots,m$, there exist linear subspaces
$L_i\in\MG(n,\alpha_i)$ such that
\begin{equation}
\label{ZiKiLi}
Z_i\subset (Z_i|L_i)+cr_{\alpha_i}(Z_i)\sqrt{\varepsilon} \,B^n
\mbox{ and }
r_{\alpha_i}(Z_i|L_i)\geq \frac{r_{\alpha_i}(Z_i)}n>0\mbox{ \ for $c=2^{\frac{n}2}n^{\frac92}$}.
\end{equation}
It follows from \eqref{ZiKiLi} and Lemma~\ref{anglesubspaces} that
$$
[L_1,\ldots,L_m]\geq 1-n^{5}2^{\frac{n}2}n^{\frac92}\sqrt{\varepsilon}\geq 1-n^{10}2^{\frac{n}2}\sqrt{\varepsilon},
$$
if $2^{\frac{n}2}n^{\frac92}\sqrt{\varepsilon}\le 1$. But if this is not satisfied, then the requested lower bound is trivially satisfied.
Thus we conclude  Theorem~\ref{ReverseAlexandrovFenchelZstab}. \hfill$\Box$
\end{proof}

\section{The stability version of Theorem~\ref{ReverseAlexandrovFenchelKZ}}
\label{sec:5}

This section is devoted to the proof of the following theorem.

\begin{theo}
\label{ReverseAlexandrovFenchelKZstab}
Let $\varepsilon\in[0,n^{-16}2^{-2n}]$. Let  $K_1\subset \R^n$ be a compact convex set  and let $Z_2,\ldots,Z_{m}\subset\R^n$ be zonoids for $2\leq m\leq n$. If
$\alpha_1,\ldots,\alpha_m\geq 1$ are integers with $\alpha_1+\cdots+\alpha_m=n$  and
$$
0< V_{\alpha_{1}}(K_1)\prod_{i=2}^{m}V_{\alpha_i}(Z_i)\le (1+\varepsilon)
\binom{n }{ \alpha_1,\ldots, \alpha_m}
V(K_1[\alpha_1],Z_2[\alpha_2],\ldots,Z_{{m}}[\alpha_{m}]),
$$
then there exist linear subspace $L_i\in\MG(n,\alpha_i)$ for $i=1,\ldots,m$
and $q\in K_1$  such that
$$[L_1,\ldots,L_m]\geq 1-n^{13}2^{2n+4} \varepsilon^{\frac{1}{8}}$$
and
\begin{eqnarray*}
K_1-q&\subset & L_1+n^{\frac{7}{2}}2^{\frac{n+2}{2}} \,r_{\alpha_1}(K_1) \,\varepsilon^{\frac{1}{2}}\,B^n,\\
Z_i&\subset & L_i+n^{8}2^{2n+4} \,r_{\alpha_i}(Z_i)\,\varepsilon^{\frac{1}{8}}\,B^n,\quad i=2,\ldots,m.
\end{eqnarray*}
\end{theo}

We prepare the proof of Theorem~\ref{ReverseAlexandrovFenchelKZstab} with two observations,
Lemma~\ref{ballintrinsicest} and Lemma~\ref{CMintrinsic}, where the first one is based on basic properties of
 the $\Gamma$ function $\Gamma$.

\begin{lemma}
\label{ballintrinsicest}
If $n\geq 2$ and $j=1,\ldots,n$, then
$$
\frac{V_j(B^n)}{\kappa_j}\leq 2^{\frac{n}{2}}.
$$
\end{lemma}

\begin{proof}
We repeatedly use the duplication formula for the Gamma function
$$
\frac{\Gamma(2x)}{\Gamma(x)}=\frac{2^{2x-1}}{\sqrt{\pi}}\cdot \Gamma\left(x+\frac12\right)\quad \mbox{  for $x>0$}.
$$
By the logarithmic convexity of the Gamma function, which yields
$$
\Gamma\left(\frac{x+y}2\right)\leq \sqrt{\Gamma(x)\Gamma(y)}\quad \text{ for $x,y>0$,}
$$
 the functional equation  $\Gamma(x+1)=x\Gamma(x)$,
 $\binom{n}{j}=\frac{\Gamma(n+1)}{\Gamma(n-j+1)\Gamma(j+1)}$
and $\kappa_d=\frac{\pi^{d/2}}{\Gamma(\frac{d}2+1)}$,  we get
\begin{align*}
\frac{V_j(B^n)}{\kappa_j}&=\frac{\binom{n}{ j}\kappa_n}{\kappa_j\kappa_{n-j}}=\frac{\Gamma(n+1)}{\Gamma(\frac{n+2}2)}\cdot
\frac{\Gamma(\frac{j+2}{2})}{\Gamma(j+1)}\cdot \frac{\Gamma(\frac{n-j+2}{2})}{\Gamma(n-j+1)}\\
&=\sqrt{\pi}\cdot\frac{\Gamma(\frac{n}2+\frac12)}{\Gamma(\frac{n-j}2+\frac12)
\Gamma(\frac{j}2+\frac{1}{2})}\leq
\sqrt{\pi}\cdot\frac{\Gamma(\frac{n}2+\frac12)}{\Gamma(\frac{n}4+\frac12)^2}\\
&=\frac{\Gamma(\frac{n}2+\frac12)}{\Gamma(\frac{n}2+1)}
\cdot\frac{\Gamma(\frac{n}4+1)}{\Gamma(\frac{n}4+\frac12)}
\cdot\frac{\sqrt{\pi}\cdot\Gamma(\frac{n}2+1)}{\Gamma(\frac{n}4+\frac12)\Gamma(\frac{n}4+1)}
=2^{\frac{n}2}\cdot\frac{\Gamma(\frac{n}2+\frac12)}{\Gamma(\frac{n}2+1)}
\cdot\frac{\Gamma(\frac{n}4+1)}{\Gamma(\frac{n}4+\frac12)}\\
&\leq 2^{\frac{n}2}\cdot\frac{\sqrt{\Gamma(\frac{n}2)\Gamma(\frac{n}2+1)}}{\Gamma(\frac{n}2+1)}
\cdot\frac{\sqrt{\Gamma(\frac{n}4+\frac12)\Gamma(\frac{n}4+\frac32)}}{\Gamma(\frac{n}4+\frac12)}
=2^{\frac{n-1}2}\sqrt{\frac{n+2}n}\leq 2^{\frac{n}2},
\end{align*}
which proves Lemma~\ref{ballintrinsicest}.
\end{proof}

\begin{lemma}
\label{CMintrinsic}
If $1\leq\alpha<d\leq n$, $0<\eta<\frac1{\alpha n}$ and $M$ is a  compact convex set in $\R^d$ such that there exists
an affine subspace $A\in\MA(n,\alpha)$ such that
$$
M\subset A+\eta r_{\alpha}(M) \,B^d\quad
\mbox{ and }\quad
r_{\alpha}(M|A)\geq \frac{r_{\alpha}(M)}n,
$$
then
$$
V_\alpha (M)\leq (1+n2^{d+1}\eta)V_\alpha (M|A).
$$
\end{lemma}

\begin{proof}
We may assume that $o\in A$. We set $C=M|A$ and $B^\alpha=A\cap B^d$,
and as before we  write $v^{(\alpha)}(\cdot )$
to denote the mixed volume in $A$.

Lemma~\ref{CMintrinsic} is  equivalent to the inequality
$$
V(M[\alpha],B^d[d-\alpha])\leq (1+n2^{d+1}\eta)V(C[\alpha],B^d[d-\alpha]).
$$
Since $B^d$ is a zonoid, using  Theorem~\ref{ReverseAlexandrovFenchelKZ} for the second  and
 $ r_{\alpha}(M)\leq n r_{\alpha}(C)$ for the third inequality, we obtain that
\begin{align*}
&V(M[\alpha],B^d[d-\alpha])\\
&\leq V((C+\eta r_{\alpha}(M)B^d)[\alpha],B^d[d-\alpha])\\
&= V(C[\alpha],B^d[d-\alpha])
 +\sum_{j=1}^{\alpha}\binom{\alpha}{ j}(\eta r_{\alpha}(M))^jV(C[\alpha-j],B^d[j],B^d[d-\alpha])\allowdisplaybreaks\\
&\leq V(C[\alpha],B^d[d-\alpha])+\sum_{j=1}^{\alpha}\frac{\binom{\alpha}{ j}}{
\binom{d }{ \alpha-j,j,d-\alpha}}(\eta r_{\alpha}(M))^j\kappa_j V_{\alpha-j}(C)
V_{d-\alpha}(B^d)\allowdisplaybreaks\\
&\leq V(C[\alpha],B^d[d-\alpha])
 +V_{d-\alpha}(B^d)\sum_{j=1}^{\alpha}\frac{\binom{\alpha}{ j}^2}{\binom{d }{ \alpha-j,j,d-\alpha}}(\eta nr_{\alpha}(C))^j
v^{(\alpha)}(C[\alpha-j],B^\alpha[j])\allowdisplaybreaks\\
&= V(C[\alpha],B^d[d-\alpha]) +\frac{V_{d-\alpha}(B^d)}{\binom{d }{ \alpha}}\sum_{j=1}^{\alpha}\binom{\alpha}{ j}(\eta n)^j
v^{(\alpha)}(C[\alpha-j],r_{\alpha}(C)B^\alpha[j])\allowdisplaybreaks\\
&\le V(C[\alpha],B^d[d-\alpha]) +\frac{V_{d-\alpha}(B^d)}{\binom{d }{ \alpha}}V_\alpha(C)\left[(1+\eta n)^\alpha-1\right],
\end{align*}
where for the final inequality we used that a translate of
$r_{\alpha}(C)B^\alpha$ is contained in $C$.

Since $n\eta<\frac1{\alpha}$ implies that $(1+n\eta)^\alpha<1+2\alpha\,n\eta$, and
$$
\frac{V_{d-\alpha}(B^d)}{\binom{d}{ \alpha}}V_\alpha(C)
=\frac{V_{d-\alpha}(B^d)}{\kappa_{d-\alpha}}V(C[\alpha],B^d[d-\alpha])\le
2^{\frac{d}{2}}V(C[\alpha],B^d[d-\alpha]),
$$
where also Lemma~\ref{ballintrinsicest} was used,
we finally get
\begin{align*}
V(M[\alpha],B^d[d-\alpha])
&\leq V(C[\alpha],B^d[d-\alpha])\left(1+2^{\frac{d}{2}} \cdot 2\alpha n \eta\right)\\
&\leq \left(1+n2^{d+1}\eta\right
)V(C[\alpha],B^d[d-\alpha]),
\end{align*}
which proves the lemma.
\end{proof}

\begin{proof}[Proof of Theorem~\ref{ReverseAlexandrovFenchelKZstab}]
  According to Proposition~\ref{oneKstab}, possibly after translating $K_1$,
there exists a linear  subspace $L_1\in\MG(n,\alpha_1)$ such that
\begin{equation}
\label{KmLmcond}
K_{1}\subset L_1+c\sqrt{\varepsilon} r_{\alpha_1}(K_1) B^n\quad
\mbox{ and }\quad
r_{\alpha_1}(K_1|L_1)\geq r_{\alpha_1}(K_1)/n
\end{equation}
where $c=2^{\frac{n+2}2}n^{\frac72}$. We set
$$
C_1=K_1|L_1 \quad
\mbox{  and }\quad B^{\alpha_1}=L_1\cap B^n.
$$
Note that \eqref{KmLmcond} implies that
\begin{equation}
\label{KmCmcond}
K_{1}\subset C_1+c\sqrt{\varepsilon} r_{\alpha_1}(K_1) B^n.
\end{equation}
For $\mathcal{Z}=(Z_2[\alpha_2],\ldots,Z_{{m}}[\alpha_{m}])$, we deduce from \eqref{KmCmcond} that
\begin{align}
V(K_1[\alpha_1],\mathcal{Z})
&\leq V((C_1+c\sqrt{\varepsilon} r_{\alpha_1}(K_1)B^n)[\alpha_1], \mathcal{Z})\nonumber\\
&= V(C_1[\alpha_1],\mathcal{Z})
\label{calZCmj}
 +\sum_{j=1}^{\alpha_1}\binom{\alpha_1}{ j}(c\sqrt{\varepsilon} r_{\alpha_1}(K_1))^jV(C_1[\alpha_1-j], B^n[j],\mathcal{Z}).
\end{align}
Next we provide several upper and lower bounds for the individual expressions involved in \eqref{calZCmj}.
The conditions in Theorem~\ref{ReverseAlexandrovFenchelKZstab} and  $1-\varepsilon\le (1+\varepsilon)^{-1}$ yield
\begin{equation}
\label{calZKmest}
(1-\varepsilon)V_{\alpha_1}(K_1)\prod_{i=1}^{m}V_{\alpha_i}(Z_{i})\leq \binom{n}{ \alpha_1,\ldots ,\alpha_m}V(K_1[\alpha_1],\mathcal{Z}).
\end{equation}
Setting $\beta=n-\alpha_1$, it follows from an application of \cite[Theorem 5.3.1]{Sch14} that
\begin{align}
\nonumber
\binom{n }{ \alpha_1,\ldots, \alpha_m}
V(C_1[\alpha_1],\mathcal{Z})&=
V_{\alpha_1}(C_1) \binom{\beta }{ \alpha_{2},\ldots, \alpha_{m}} \\
\label{calZCmmest}
&\qquad \times v^{(\beta)}\left((Z_{2}|L_1^\bot)[\alpha_{1}],\ldots,(Z_{m}|L_1^\bot)[\alpha_{m}]\right).
\end{align}
An application of Theorem~\ref{ReverseAlexandrovFenchelKZ} yields that
\begin{equation}
\label{calZCmjest1}
\binom{n }{\alpha_1,\ldots, \alpha_m}\binom{\alpha_1}{j} V(C_1[\alpha_1-j], B^n[j],\mathcal{Z})
\leq \kappa_j V_{\alpha_1-j}(C_1)\prod_{i=2}^{m}V_{\alpha_i}(Z_{i}).
\end{equation}
In addition, using first \eqref{KmLmcond}, then Lemma~\ref{ballintrinsicest}
and finally the fact that a translate of $r_{\alpha_1}(C_1)B^{\alpha_1}$ lies in $C_1$, we get
\begin{align}
\nonumber
 \binom{\alpha_1}{ j}^{-1}r_{\alpha_1}(K_1)^j\kappa_jV_{\alpha_1-j}(C_1)
&\leq n^j \cdot r_{\alpha_1}(C_1)^j\cdot v^{(\alpha_1)}(C_1[\alpha_1-j],B^{\alpha_1}[j])\\
\nonumber
&= n^j v^{(\alpha_1)}(C_1[\alpha_1-j],r_{\alpha_1}(C_1)B^{\alpha_1}[j])\\
\label{calZCmjest2}
&\le  n^j  V_{\alpha_1}(C_1).
\end{align}

Combining \eqref{calZCmjest1} and \eqref{calZCmjest2}, we get
\begin{equation}\label{calZCmjest}
\binom{n }{ \alpha_1,\ldots, \alpha_m}r_{\alpha_1}(K_1)^jV(C_1[\alpha_1-j], B^n[j],\mathcal{Z})
\le  n^j  V_{\alpha_1}(C_1)\prod_{i=2}^{m}V_{\alpha_i}(Z_{i}).
\end{equation}
Substituting \eqref{calZKmest},
\eqref{calZCmmest} and \eqref{calZCmjest} into  \eqref{calZCmj} and using  $V_{\alpha_1}(C_1)\le V_{\alpha_1}(K_1)$ by  Lemma~\ref{proj}, we finally arrive at
(with $\beta=n-\alpha_1$)
\begin{align*}
(1-\varepsilon)\prod_{i=2}^{m}V_{\alpha_i}(Z_{i})&\leq
 \binom{\beta}{ \alpha_{2},\ldots, \alpha_{m}}
  v^{(\beta)}\left((Z_{2}|L_1^\bot)[\alpha_{2}],\ldots,(Z_{m}|L_1^\bot)[\alpha_{m}]\right)\\
&\quad + \left(\prod_{i=2}^{m}V_{\alpha_i}(Z_{i})\right)
\sum_{j=1}^{\alpha_1}\binom{\alpha_1}{ j}(cn\sqrt{\varepsilon})^j.
\end{align*}
Since $\varepsilon\le (5cn^2)^{-2}$, we have $cn\sqrt{\varepsilon}<\frac1n$, and hence
$$
\sum_{j=1}^{\alpha_1}\binom{\alpha_1}{j}
(cn\sqrt{\varepsilon})^j=\left(1+cn\sqrt{\varepsilon}\right)^{\alpha_1}-1\leq 2cn^2\sqrt{\varepsilon}.
$$
Using that $1-\varepsilon-2cn^2\sqrt{\varepsilon}\ge 1-2.5\,cn^2\sqrt{\varepsilon}$, we deduce that
\begin{equation}
\label{ZKalmostfinal}
0<(1-2.5\,cn^2\sqrt{\varepsilon}) \prod_{i=2}^{m}V_{\alpha_{i}}(Z_i)\leq
  \binom{\beta }{ \alpha_{2},\ldots, \alpha_{m}} v^{(\beta)}\left((Z_{2}|L_1^\bot)[\alpha_{2}],\ldots,(Z_{m}|L_1^\bot)[\alpha_{m}]\right).
\end{equation}
In particular, this shows that $V_{\alpha_i}(Z_{i}|L_1^\bot)>0$ for $i=2,\ldots,m$. Moreover, from \eqref{ZKalmostfinal},
$V_{\alpha_i}(Z_{i}|L_1^\bot)\leq V_{\alpha_{i}}(Z_i)$ for $i=2,\ldots,m$ and $(1-2.5\,cn^2\sqrt{\varepsilon})^{-1}\le 1+5\,cn^2\sqrt{\varepsilon}$ we now obtain that
\begin{align}
0<\prod_{i=2}^{m}V_{\alpha_i}(Z_{i}|L_1^\bot)&\leq  \prod_{i=2}^{m}V_{\alpha_i}(Z_{i})\label{ZiLmbot}\\
&\leq  (1+5cn^2\sqrt{\varepsilon})
\binom{\beta }{ \alpha_{2},\ldots, \alpha_{m}}
v^{(\beta)}\left((Z_{2}|L_1^\bot)[\alpha_{2}],\ldots,(Z_{m}|L_1^\bot)[\alpha_{m}]\right).
\label{ZiLmbot2}
\end{align}
We set $\widetilde{Z}_i=Z_{i}|L_1^\bot$ for $i=2,\ldots,m$. Based on \eqref{ZiLmbot} and \eqref{ZiLmbot2},   we apply
 Proposition~\ref{oneKstab} in $L_1^\bot$  if $m\ge 3$ (and hence $2\le m-1\le\alpha_2+\cdots+\alpha_m=n-\alpha_1$), and directly apply Proposition~\ref{projstab} and
 Lemma~\ref{rm+1A} in $L_1^\bot$  if $m=2$. Thus we obtain
$\alpha_i$-dimensional linear subspaces $\widetilde{L}_i\subset L_1^\bot$ for $i=2,\ldots,m$ such that
$$
r_{\alpha_i}(\widetilde{Z}_i|\widetilde{L}_i)\geq
 r_{\alpha_i}(\widetilde{Z}_i)/n\mbox{ \ \ and \ \ }\widetilde{Z}_i\subset
\widetilde{L}_i+2^{n+2}n^{7}\varepsilon^{\frac{1}{4}}r_{\alpha_i}(\widetilde{Z}_i) \,(B^n\cap L_1^\bot),
$$
where we used that
$c\sqrt{5cn^2\sqrt{\varepsilon}}\le 2^{n+2}n^7\varepsilon^{\frac{1}{4}}$.

We deduce from Lemma~\ref{CMintrinsic} that if $i=2,\ldots,m$, then
\begin{equation}\label{Bound1}
V_\alpha (\widetilde{Z}_i)\leq (1+n2^{d+1}2^{n+2}n^{7}\varepsilon^{\frac14})V_\alpha (\widetilde{Z}_i|\widetilde{L}_i)
\leq (1+n^{8}2^{2n+3}\varepsilon^{\frac14})V_\alpha (\widetilde{Z}_i|\widetilde{L}_i).
\end{equation}
On the other hand,
\begin{equation}\label{Bound2}
V_{\alpha_i}(Z_i)\leq(1+5cn^2\sqrt{\varepsilon})V_{\alpha_i}(\widetilde{Z}_i),\quad i=2,\ldots,m,
\end{equation}
follows from \eqref{ZiLmbot2}, Theorem~\ref{ReverseAlexandrovFenchelKZ}
and again from $V_{\alpha_j}(\widetilde{Z}_j)\leq V_{\alpha_j}(Z_j)$ for $j=2,\ldots,m$.

 From $\varepsilon\le (5cn^2)^{-2}$, \eqref{Bound1}, \eqref{Bound2} and  $\widetilde{Z}_i|\widetilde{L}_i=Z_i|\widetilde{L}_i$ for  $i=2,\ldots,m$ (note
 that $\widetilde{L}_i\subset L_1^\bot$) we deduce that
\begin{align*}
V_{\alpha_i}(Z_i)&\leq(1+5cn^2\sqrt{\varepsilon})V_{\alpha_i}(\widetilde{Z}_i)\leq
 (1+5cn^2\sqrt{\varepsilon})(1+n^{8}2^{2n+3}\varepsilon^{\frac14})V_\alpha (\widetilde{Z}_i|\widetilde{L}_i)\\
 &\leq
(1+n^{8}2^{2n+5}\varepsilon^{\frac14})V_{\alpha_i}(Z_i|\widetilde{L}_i).
\end{align*}
It follows from Proposition~\ref{projstab} that
$$
r_{\alpha_i+1}(Z_i)\leq \sqrt{2^{n+2}n^5n^{8}2^{2n+5}\varepsilon^{\frac14}}\cdot r_{\alpha_i}(Z_i)
\leq n^{7}2^{2n+4}\varepsilon^{\frac18}\cdot r_{\alpha_i}(Z_i),
$$
and in turn, Lemma~\ref{rm+1A} (ii) yields that there exists a linear subspace $L_i\in\MG(n,\alpha_i) $
such that
\begin{equation}
\label{ZiwhenK}
Z_i\subset L_i+ n^{8}2^{2n+4}\varepsilon^{\frac18} r_{\alpha_i}(Z_i)\cdot B^n
\mbox{ \ \ and \ \ }r(Z_i|L_i)\geq r_{\alpha_i}(Z_i)/n>0.
\end{equation}
Given \eqref{KmLmcond} and \eqref{ZiwhenK}, we conclude
$$
[L_1,\ldots,L_m]\geq 1-n^5\cdot n^{8}2^{2n+4}\varepsilon^{\frac18}=
1-  n^{13}2^{2n+4}\varepsilon^{\frac18}
$$
by Lemma~\ref{anglesubspaces}.
\end{proof}

\medskip

\noindent
\textbf{Acknowledgement.} K.~B\"or\"oczky was supported by research grant NKFIH 132002.

Daniel Hug was supported by research grant HU 1874/5-1 (DFG).


\begin{thebibliography}{99}





\bibitem{Artstein}
S.~Artstein-Avidan, D.~Florentin, Y.~Ostrover:
Remarks about mixed discriminants and volumes.
Communications in Contemporary Mathematics 16(2) (2014), 1350031 (14 pages).

\bibitem{Ball91}
K.~Ball: Volume ratios and a reverse isoperimetric inequality. J. London Math. Soc. (2) 44 (2) (1991), 351--359.

\bibitem{Ball91b}
K.~Ball: Shadows of convex bodies. Trans. Amer. Math. Soc. 327 (2) (1991), 891--901.

\bibitem{Ball2003}
K.~Ball: Convex geometry and functional analysis. Handbook of the geometry of Banach spaces, vol.~1
(eds W.~B. Johnson and L.~Lindenstrauss, Elsevier, Amsterdam, 2003), 161--194.

\bibitem{BaJ17}
M.~Barchiesi, V.~Julin:
Robustness of the Gaussian concentration inequality and the Brunn-Minkowski inequality.
Calc. Var. Partial Differential Equations, 56 (2017), Paper No. 80, 12 pp.

\bibitem{Bartha}
F.A.~Bartha, F.~Bencs, K.J.~B\"or\"oczky, D.~Hug: Extremizers and stability of the Betke--Weil inequality. https://arxiv.org/pdf/:2103.11672.pdf

\bibitem{Barthe1}
F.~Barthe: An extremal property of the mean width of the simplex. Math. Ann. 310 (1998), 685--693.


\bibitem{Barthe2}
F.~Barthe: On a reverse form of the Brascamp-Lieb inequality. Invent. Math. 134 (1998), 335--361.

\bibitem{Behrend}
F.~Behrend: \"Uber einige Affininvarianten konvexer Bereiche. Math. Ann. 113 (1) (1937), 713--747.



\bibitem{BW}
U.~Betke, W.~Weil: Isoperimetric inequalities for the mixed area of plane convex sets. Arch. Math. 57 (1991), 501--507.

\bibitem{BH17}
K.~Böröczky, D.~Hug. Isotropic measures and stronger forms of the reverse isoperimetric inequality. Trans. Amer. Math. Soc. 369(10) (2017), 6987--7019.

\bibitem{BH}
K.J.~B\"or\"oczky, D.~Hug: A reverse Minkowski-type inequality. Proc. Amer. Math. Soc. 148 (2020), 4907--4922.


\bibitem{BFH19}
K.~Böröczky, F.~Fodor, D.~Hug. Strengthened volume inequalities for $L_p$ zonoids of even isotropic measures. Trans. Amer. Math. Soc. 371 (2019), 505--548.

\bibitem{BFH21}
K.~Böröczky, F.~Fodor, D.~Hug. Strengthened inequalities for the mean width and the $l$-norm. J. London Math. Soc. 0 (2021), 1--36.

\bibitem{Brazitikos}
S.~Brazitikos, F.~McIntyre: Vector-Valued Maclaurin in\-equali\-ties. Preprint (2021).  \newline https://arxiv.org/pdf/2102.05900.pdf

\bibitem{Chr12}
M.~Christ:
An approximate inverse Riesz–Sobolev inequality, preprint, available online at \newline 
http://arxiv.org/abs/1112.3715, 2012.

\bibitem{CLMS}
D.~Cordero-Erausquin, B.~Klartag, Q.~Merigot, F.~Santambrogio: One more proof of the Alexandrov-Fenchel inequality. C.~R. Math. Acad. Sci. Paris 357 (2019), no. 8, 676--680.

\bibitem{Dis73}
V.I.~Diskant:
 Stability of the solution of a Minkowski equation. (Russian)
Sibirsk. Mat. \v Z. 14 (1973), 669--673.
[Eng. transl.: Siberian Math. J. 14 (1974), 466--473.]

\bibitem{ElK14}
R.~Eldan, B.~Klartag:
Dimensionality and the stability of the Brunn-Minkowski inequality.
Annali della Scuola Normale Superiore di Pisa, Classe di Scienze (5) 13 (2014), 975--1007.

\bibitem{FiJ17}
A.~Figalli, D.~Jerison:
Quantitative stability for the Brunn-Minkowski inequality.
Adv. Math. 314 (2017), 1--47.


\bibitem{FMP09}
A.~Figalli, F.~Maggi, A.~Pratelli:
A refined Brunn-Minkowski inequality for convex sets.
Annales de IHP  26 (2009), 2511--2519.

\bibitem{FMP10}
A.~Figalli, F.~Maggi, A.~Pratelli:
A mass transportation approach to quantitative
isoperimetric inequalities.  Invent. Math. 182 (2010), 167--211.

\bibitem{Gro88}
H.~Groemer:
 On the Brunn-Minkowski theorem.
Geom. Dedicata 27 (1988), 357--371.

\bibitem{Gro93}
H.~Groemer:
Stability of geometric inequalities.
In: Handbook of convex geometry (P.M.~Gruber, J.M.~Wills, eds),
North-Holland, Amsterdam, 1993, 125--150.

\bibitem{Gru07}
P.M.~Gruber: Convex and Discrete Geometry. Springer, 2007.

\bibitem{Hadwiger}
H.~Hadwiger: Vorlesungen \"uber Inhalt, Oberfl\"ache und Isoperimetrie. Springer, Berlin, 1957.

\bibitem{HST}
P.~van Hintum, H.~Spink, M.~Tiba:
Sharp Stability of Brunn-Minkowski for Homothetic Regions.
Journal EMS, accepted.
arXiv:1907.13011

\bibitem{HST2}
P.~van Hintum, H.~Spink, M.~Tiba:
Sharp quantitative stability of the planar Brunn-Minkowski inequality.
arXiv:1911.11945

\bibitem{HW2020}
D.~Hug, W.~Weil: Lectures on Convex Geometry. Graduate Texts in Mathematics. Vol. 286. Springer Nature Switzerland AG, 2020.

\bibitem{HS11}
D.~Hug, R.~Schneider: Reverse inequalities for zonoids and their application.
Adv. Math. 228 (2011), 2634--2646.

\bibitem{KoM}
A.V.~Kolesnikov, E.~Milman:
Local $L_p$-Brunn-Minkowski inequalities for $p<1$.
Memoirs AMS, accepted.
arXiv:1711.01089

\bibitem{Leicht}
K.~Leichtwei{\ss}: Konvexe Mengen. Hochschultext. Springer, Berlin, 1980.

\bibitem{Li}
A.-J.~Li, G.~Leng: Mean width inequalities for isotropic measures. Math. Z. 270 (2012), 1089--1110.

\bibitem{Livshyts}
G.~Livshyts: On a conjectural symmetric version of Erhard's inequality. \newline https://people.math.gatech.edu/~glivshyts6/Liv-arxiv.pdf

\bibitem{Lutwak1}
E.~Lutwak: Selected affine isoperimetric inequalities. In: Handbook of convex geometry, North-Holland, Amsterdam,  1993, 151--176.

\bibitem{Lutwak2}
E.~Lutwak, D.~Yang, G.~Zhang: Volume inequalities for subspaces of $L_p$. J. Diff. Geom. 68 (2004), 159--184.


\bibitem{Lutwak3}
E.~Lutwak, D.~Yang, G.~Zhang: $Lp$ John ellipsoids. Proc. London Math. Soc. 90 (2005), 497--520.

\bibitem{Lutwak4}
E.~Lutwak, D.~Yang, G.~Zhang:  Volume inequalities for isotropic measures. Amer. J. Math. 129 (2007), 1711--1723.

\bibitem{MaM17}
Y.~Martinez-Maure:
A stability estimate for the Aleksandrov-Fenchel inequality under regularity assumptions.
Monatsh. Math. 182 (2017), 65--76.


\bibitem{MilmanE2020}
E.~Milman, A.~Yehudayoff: Sharp isoperimetric inequalities for affine Quermassintegrals.\newline 
https://arxiv.org/pdf/2005.04769.pdf

\bibitem{PPV}
G.~Paouris, P.~Pivovarov, P.~Valettas: On a quantitative reversal of Alexandrov's inequality. Trans. Amer. Math. Soc. 371 (2019), no. 5, 3309--3324.

\bibitem{SSZ}
C.~Saroglou, I.~Soprunov, A.~Zvavitch: Wulff shapes and a characterization of simplices via a Bezout type inequality. Adv. Math. 357 (2019), 106789, 24 pp.

\bibitem{Schmuck}
M.~Schmuckenschl\"ager. An extremal property of the regular simplex. Convex geometric analysis (Berkeley, CA, 1996), 199--202, Math. Sci. Res. Inst. Publ., 34, Cambridge Univ. Press, Cambridge, 1999.

\bibitem{Sch85}
R.~Schneider:
On the Aleksandrov-Fenchel inequality.
In: Discrete geometry and convexity,
volume 440 of Ann. New York Acad. Sci., New York Acad. Sci., (1985), 132--141.

\bibitem{Sch90a}
R.~Schneider:
On the Aleksandrov-Fenchel inequality for convex bodies. I.
Results Math. 17 (1990), 287--295.

\bibitem{Sch90b}
R.~Schneider:
A stability estimate for the Aleksandrov--Fenchel inequality, with an application to mean curvature.
Manuscripta Math. 69 (1990), 291--300.

\bibitem{Schneider1994a}
R.~Schneider: Equality in the Aleksandrov--Fenchel inequality—present state and new results. Intuitive geometry (Szeged, 1991), 425--438, Colloq. Math. Soc. János Bolyai, 63, North-Holland, Amsterdam, 1994.

\bibitem{Schneider1994b}
R.~Schneider:  Polytopes and Brunn--Minkowski theory. Polytopes: abstract, convex and computational (Scarborough, ON, 1993), 273-–299, NATO Adv. Sci. Inst. Ser. C Math. Phys. Sci., 440, Kluwer Acad. Publ., Dordrecht, 1994.

\bibitem{Sch14}
R.~Schneider:
Convex bodies: the Brunn-Minkowski theory. Second Edition. Cambridge,  2014.

\bibitem{SW}
R.~Schneider, W.~Weil: Stochastic and Integral Geometry. Springer, 2008.

\bibitem{Seg12}
A.~Segal:
Remark on stability of Brunn-Minkowski and isoperimetric inequalities for convex bodies.
In: Geometric aspects of functional analysis, volume
2050 of Lecture Notes in Math.,  Springer, Heidelberg, 2012, 381--391.


\bibitem{SvH}
Y.~Shenfeld, R.~van Handel: Mixed volumes and the Bochner method. Proc. Amer. Math. Soc. 147 (2019), no. 12, 5385--5402.

\bibitem{SvH2}
Y.~Shenfeld, R.~van Handel: Extremals in Minkowski’s Quadratic Inequality. Duke Math. J. (to appear). https://arxiv.org/abs/1902.10029

\bibitem{SvH3}
Y.~Shenfeld, R.~van Handel: The Extremals of the Alexandrov-Fenchel Inequality for Convex Polytopes.
https://arxiv.org/abs/2011.04059



\bibitem{Wang}
X.~Wang: A remark on the Alexandrov-Fenchel inequality.
J. Funct. Anal.  274 (2018), no. 7, 2061--2088.

\bibitem{Xiao}
Xiao, Jian
Bézout-type inequality in convex geometry.
Int. Math. Res. Not. IMRN 2019, no. 16, 4950--4965.


\end{thebibliography}
\end{document}